\documentclass[11pt]{amsart}
\usepackage{mabliautoref}
\usepackage{amssymb,amsthm,amsmath}
\RequirePackage[dvipsnames,usenames]{xcolor}
\usepackage{hyperref}
\usepackage[all]{xy}
\usepackage{tikz}

\hypersetup{
bookmarks,
bookmarksdepth=3,
bookmarksopen,
bookmarksnumbered,
pdfstartview=FitH,
colorlinks,backref,hyperindex,
linkcolor=Sepia,
anchorcolor=BurntOrange,
citecolor=MidnightBlue,
citecolor=OliveGreen,
filecolor=BlueViolet,
menucolor=Yellow,
urlcolor=OliveGreen
}

\theoremstyle{remark}

\DeclareMathAlphabet{\mathchanc}{OT1}{pzc}%
                                 {m}{it}

\newcommand{\mcH}{\mathchanc{H}}

\newcommand{\mcm}{\mathchanc{m}}

\newcommand{\mco}{\mathchanc{o}}

\newcommand{\bP}{\mathbb{P}}
\newcommand{\bQ}{\mathbb{Q}}

\newcommand{\bZ}{\mathbb{Z}}

\newcommand{\scr}{\mathcal}

\newcommand{\sD}{\scr{D}}

\newcommand{\sF}{\scr{F}}

\newcommand{\sH}{\scr{H}}

\newcommand{\sJ}{\scr{J}}
\newcommand{\sK}{\scr{K}}
\newcommand{\sL}{\scr{L}}

\newcommand{\sO}{\scr{O}}

\DeclareMathOperator{\trdeg}{{trdeg}}

\DeclareMathOperator{\Sym}{{Sym}}

\DeclareMathOperator{\Tr}{Tr}
\DeclareMathOperator{\codim}{codim}

\DeclareMathOperator{\Ker}{{Ker}}

\DeclareMathOperator{\Hom}{Hom}

\newcommand{\sHom}[0]{{\mcH\mco\mcm}}

\DeclareMathOperator{\id}{{id}}

\DeclareMathOperator{\im}{{im}}

\DeclareMathOperator{\rk}{{rk}}

\DeclareMathOperator{\Spec}{{Spec}}

\DeclareMathOperator{\Ca}{{Ca}}

\newcommand{\factor}[2]{\left. \raise 2pt\hbox{\ensuremath{#1}} \right/
        \hskip -2pt\raise -2pt\hbox{\ensuremath{#2}}}

\makeatletter
\renewcommand\subsection{
  \renewcommand{\sfdefault}{pag}
  \@startsection{subsection}%
  {2}{0pt}{.8\baselineskip}{.4\baselineskip}{\raggedright
    \sffamily\itshape\small\bfseries
  }}
\renewcommand\section{
  \renewcommand{\sfdefault}{phv}
  \@startsection{section} %
  {1}{0pt}{\baselineskip}{.8\baselineskip}{\centering
    \sffamily
    \scshape
    \bfseries
}}
\makeatother

\renewcommand{\theenumi}{\arabic{enumi}}

\usepackage[left=1.02in,top=1.0in,right=1.02in,bottom=1.0in]{geometry}

\title{On subadditivity of Kodaira dimension in positive characteristic}
\author{Zsolt Patakfalvi}

\begin{document}

\maketitle

\begin{abstract}
We show that for a  surjective, separable morphism $f$ of smooth projective varieties   over a field  of positive characteristic such that $f_* \sO_X \cong \sO_Y$ subadditivity of Kodaira dimension holds, provided the base is of general type  and the Hasse-Witt matrix of the geometric general fiber is not nilpotent. 
\end{abstract}

\tableofcontents

\section{Introduction}

Subadditivity of Kodaira dimension, also known as $C_{n,m}$ conjecture, is a conjecture of Iitaka \cite{Iitaka_Genus_and_classification_of_algebraic_varieties_I}, \cite[page 279]{Ueno_Classification_of_algebraic_varieties_I}  stating that for a fiber space $f : X \to Y$ with geometric generic fiber $F$ the respective Kodaira dimensions satisfy the following inequality:
\begin{equation}
\label{eq:subadditivity_of_Kodaira_dimension_initial}
\kappa(X) \geq  \kappa(F) + \kappa(Y). 
\end{equation}

Here we prove the above conjecture in positive characteristic when the base is of general type and the Hasse-Witt matrix of the geometric general fiber is not nilpotent:

\begin{theorem}
\label{thm:subadditivity_of_Kodaira_dimension}
Let  $f : X \to Y$ be a separable, surjective morphism of smooth projective varieties over an algebraically closed field $k$ of positive characteristic such that $f_* \sO_X \cong \sO_Y$. Further assume that $\kappa(Y)=\dim Y$, and the Hasse-Witt matrix of the geometric generic fiber $F$ is not nilpotent (including that it is not the zero matrix). 
Then 
\begin{equation}
\label{eq:subadditivity_of_Kodaira_dimension}
  \kappa \left(X\right) \geq \kappa \left(K_F\right) + \kappa \left(Y \right),
\end{equation}
and therefore
\begin{equation}
\label{eq:subadditivity_of_Kodaira_dimension_resolutions}
\kappa \left(X\right) \geq \kappa \left(F\right) + \kappa \left(Y \right). 
\end{equation}
\end{theorem}

The available proofs of the characteristic zero versions of \autoref{thm:subadditivity_of_Kodaira_dimension} use either Hodge theory or Kodaira vanishing. Since these are not available in positive characteristic, the main issue is to circumvent these tools by special positive characteristic methods. First a few remarks:

\begin{remark}
\label{rem:Kodaira_dimension}
For a variety $Z$, $\kappa(Z)$ denotes the Kodaira dimension as defined in \cite[Def 5.1]{Luo_Kodaira_dimension_of_algebraic_function_fields}, while $\kappa(\omega_Z)$ denotes the Kodaira-Iitaka dimension of the canonical bundle as defined in \autoref{defn:Kodaira_Iitaka}. 
With these definitions, the inequality \autoref{eq:subadditivity_of_Kodaira_dimension} implies \autoref{eq:subadditivity_of_Kodaira_dimension_resolutions} by \autoref{cor:Kodaira_vs_Kodaira-Iitaka_of_canonical}, hence the main content of \autoref{thm:subadditivity_of_Kodaira_dimension} is that \autoref{eq:subadditivity_of_Kodaira_dimension} holds. Note that we use \cite[Def 5.1]{Luo_Kodaira_dimension_of_algebraic_function_fields} as the definition of the Kodaira dimension, since in positive characteristic resolution of singularities is not known to exist in dimension greater than 3. However, this agrees with the usual definition whenever there is a resolution by point \autoref{itm:Kodaira_vs_Kodaira-Iitaka_of_canonical:smooth} of \autoref{prop:Kodaira_vs_Kodaira-Iitaka_of_canonical}.
\end{remark}

\begin{remark}
The Hasse-Witt matrix of a variety $Z$ over $k$ is the matrix of the action of the relative Frobenius on $H^n (Z, \sO_Z)$, where $n = \dim Z$. Since $k$ is perfect this action can be  identified with that of the absolute Frobenius. The Hasse-Witt matrix of the geometric generic fiber being non nilpotent is equivalent to the condition $S^0 \left(X_\eta, \omega_{X_\eta}\right) \neq 0$, where $\eta$ is the perfect closure of the generic point of $Y$ and $S^0\left(X_\eta, \omega_{X_\eta}\right)$ is the semistable submodule of $H^0\left(X_\eta, \omega_{X_\eta}\right)$ with respect to the dual action of the Frobenius (i.e., the stable image of the iterations of the Frobenius action, see \cite[Definition 4.1]{Schwede_A_canonical_linear_system}  for the general definition of $S^0$). In this sense the condition that the Hasse-Witt matrix is non-nilpotent requires that the Frobenius stable genus of the geometric general fiber is not zero. 

Furthermore by \autoref{thm:openness_S_0} the condition $S^0 \left(X_\eta, \omega_{X_\eta} \right) \neq 0$ is equivalent to saying that there is an open set $U \subseteq Y$ such that $S^0 \left(X_y, \omega_{X_y} \right) \neq 0$ for every perfect point $y \in U$ (a perfect point is a morphism $\Spec K \to U$, where $K$ is perfect).
\end{remark}

\begin{remark}
The non-nilpotence of the Hasse-Witt matrix is the expected behavior for general varieties with non-zero genus. This can be made sense in two different ways: First, if one reduces a smooth characteristic zero variety mod $p$, then the Hasse-Witt matrix is conjectured to be invertible for  infinitely many values of $p$ \cite{Mustata_Srinivas_Ordinary_varieties_and}. Second, if $X$ is a smooth variety over a field of  positive characteristic then the general member in the moduli space of $X$ is expected to have invertible Hasse-Witt matrix. Of course the latter is not a precise conjecture since $X$ does not always have a meaningful moduli space. However, it can be made precise in particular cases (e.g., $X$ is a high enough degree hyperplane section of a Gorenstein variety \cite[XXI, Th\'eor\`eme 1.4]{Seminaire_de_Geometrie_Algebrique_du_Bois_Marie_7_2_Groupes_de_monodromie_en_geometrie_algebrique}, $X$ is a curve of genus at least two or a complete intersection in $\bP^n$
\cite[Theorem 4, Theorem 5]{Koblitz_p_adic_variation_of_the_zeta_function_over_families_of_varieties_defined_over_finite_fields})
\end{remark}


\begin{remark}
It is easy to see that by the assumption $\dim Y = \kappa(Y)$ in fact \autoref{eq:subadditivity_of_Kodaira_dimension} is an equality. Same holds for \autoref{eq:subadditivity_of_Kodaira_dimension2} below.
\end{remark}

\begin{remark}
\emph{Earlier results:}
In characteristic zero subadditivity of Kodaira dimension has been a major driving force in the development of higher dimensional algebraic geometry  
\cite{Ueno_Classification_of_algebraic_varieties_I,Ueno_On_algebraic_fibre_spaces_of_abelian_varieties,Viehweg_Canonical_divisors_and_the_additivity_of_the_Kodaira_dimension_for_morphisms_of_relative_dimension_one,Kawamata_Characterization_of_abelian_varieties,Kawamata_The_Kodaira_dimension_of_certain_fiber_spaces_sketch,Kawamata_Kodaira_dimension_of_algebraic_fiber_spaces_over_curves,Kawmata_Hodge_theory_and_Kodaira_dimension,Kawamata_Kodaira_dimension_of_algebraic_fiber_spaces_over_curves,Kawamata_Minimal_models_and_the_Kodaira_dimension_of_agebraic_fiber_spaces,Viehweg_Weak_positivity,Viehweg_Weak_positivity_and_the_additivity_of_the_Kodaira_dimension_II_The_local_Torelli_map,Kollar_Subadditivity_of_the_Kodaira_dimension,Fujino_Algebraic_fiber_spaces_whose_general_fibers_are_of_maximal_Albanese_dimension,Birkar_The_Iitaka_conjecture_C_n_m_in_dimension_six,Lai_Varieties_fibered_by_good_minimal_models,Fujino_On_maximal_Albanese_dimensional_varieties,Chen_Hacon_Kodaira_dimension_of_irregular_varieties}.  
In particular, 
using the notations of \autoref{eq:subadditivity_of_Kodaira_dimension_initial}, the conjecture is proven if either $Y$ or $F$ are of general type or of maximal Albanese dimension. It is also shown if the general fiber of the Iitaka fibration of $F$ admits a good minimal model (i.e. including abundance). This latter also includes many low-dimensional known cases, except two: when $\dim Y = 1$ or $\dim X=6$.

In positive characteristic a special case was shown by the author in \cite[Corollary 4.6]{Patakfalvi_Semi_positivity_in_positive_characteristics} when both $Y$ and $F$ are of general type and further $K_{X/Y}$ is $f$-semi-ample. So, this special case requires the relative minimal model program in positive characteristic to imply a general subadditivity of Kodaira dimension type statement. Recently in \cite{Chen_Zhang_The_subadditivity_of_the_Kodaira-dimension_for_fibrations_of_relative_dimension_one_in_positive_characteristics} the  subadditivity conjecture was shown for fibrations of relative dimension one. 
\end{remark}

In fact, we are proving a slightly more general statement than that of Theorem \ref{thm:subadditivity_of_Kodaira_dimension}.

\begin{theorem}
\label{cor:subadditivity_of_Kodaira_dimension}
Let  $f : X \to Y$ be a separable, surjective morphism of projective varieties over a perfect field $k$ of positive characteristic such that $f_* \sO_X \cong \sO_Y$. Further assume that $Y$ is regular, $X$ is normal and Gorenstein,  $\kappa(Y)=\dim Y$ and 
 $S^0(X_\eta, \omega_{X_\eta}) \neq 0$, where $\eta$ is the perfect closure of the generic point of $Y$. Then 
\begin{equation}
\label{eq:subadditivity_of_Kodaira_dimension2}
  \kappa \left(K_X\right) \geq \kappa \left(K_{X_\eta}\right) + \kappa \left(K_Y \right) .
\end{equation}
\end{theorem}

\begin{remark}
So, \autoref{cor:subadditivity_of_Kodaira_dimension} implies \autoref{thm:subadditivity_of_Kodaira_dimension} indeed, because if $k$ is algebraically closed and $X$ and $Y$ are smooth, then $\kappa(X)=\kappa(K_X)$ and $\kappa(Y)= \kappa(K_Y)$ as explained in \autoref{rem:Kodaira_dimension}. Further, by \autoref{cor:Kodaira_Iitaka_dimension_base_extension}), the Kodaira-Iitaka dimension of the canonical bundles of the geometric generic, perfect generic and generic fibers are the same. 
\end{remark}

\autoref{thm:subadditivity_of_Kodaira_dimension} and  \autoref{cor:subadditivity_of_Kodaira_dimension} are very much reminiscent of the characteristic zero statement \cite[Theorem 3]{Kawamata_Characterization_of_abelian_varieties}. The reason is that both rely on certain semi-posivity of  sheaves of the type $f_* \omega_{X/Y}^m$. In characteristic zero this was  made explicit in \cite{Viehweg_Weak_positivity}, where it was shown that $f_* \omega_{X/Y}^m$ is weakly-positive for every integer $m >0$ and surjective morphism of smooth varieties $X \to Y$. Unfortunately, there is no chance for this to hold in positive characteristic, because there are examples of families for which $f_* \omega_{X/Y}$ is not semi-positive \cite[3.2]{Moret_Bailly_Familles_de_courbes_et_de_varietes_abeliennes_sur_P_1_II_exemples}. 
On the other hand there is a subsheaf of $S^0 f_* \omega_{X/Y} \subseteq f_* \omega_{X/Y}$ \cite[Definition 2.14]{Hacon_Xu_On_the_three_dimensional_minimal_model_program_in_positive_characteristic} which  in certain sense captures the Hodge theoretically nicely behaving sections \cite[Definition 4.1, Prop 5.3]{Schwede_A_canonical_linear_system}. Unfortunately, this subsheaf can be accidentally too big, i.e., it captures not only the sections that are nice on every fiber, but also some that are accidentally nice globally over an affine open set. So, in particular, we can only prove weak-positivity of $S^0 f_{Y^n,*} \omega_{X_{Y^n}/Y^n}$ for some $n \gg 0$ in \autoref{prop:weak_positivity}, where $Y^n$ is the source of the $n$-th iterated absolute Frobenius of $Y$. Unfortunately, this is still not enough for our purposes. We would need the same statement  for $\omega_{X_{Y^n}/Y^n}$ replaced by $\omega_{X_{Y^n}/Y^n}^m$. In characteristic zero, this follows from the $m=1$ case by a cyclic covering trick. 
Unfortunately cylic covers behave very differently with respect to $S^0$ than to $H^0$. This is the main reason why we are actually unable to deduce the $m>1$ case. 
So, instead of deducing the $m>1$ case of the above weak-positivity statement we follow another path, which can be thought of as proving a weaker than weak-positivity. In fact, even proving weak-positivity of $S^0 f_{Y^n,*} \omega_{X_{Y^n}/Y^n}$ in \autoref{prop:weak_positivity} is not necessary for our argument. We include it only because the proof is short and we think it is an interesting statement in itself. Instead we prove the following theorem, from which \autoref{cor:subadditivity_of_Kodaira_dimension} follows by a short argument.

\begin{theorem}
\label{thm:effective_general_fiber_implies_inf_is_zero}
Let  $f : X \to Y$ be a separable, surjective morphism of projective varieties over a perfect field $k$ of positive characteristic such that $f_* \sO_X \cong \sO_Y$. Further assume that $X$ is normal, Gorenstein, $Y$ is  regular
and $S^0(X_\eta, \omega_{X_\eta}) \neq 0$, where $\eta$ is the perfect closure of the generic point of $Y$. Fix also an ample Cartier divisor $L$ on $Y$. Then
\begin{equation*}
\inf \{ s \in \bQ | \kappa(K_{X/Y} + s f^* L) \geq 0  \} \leq 0. 
\end{equation*}
\end{theorem}

Note that the set $\{ s \in \bQ | \kappa(K_{X/Y} + s f^* L) \geq 0  \}$ is an interval of type $[a,\infty) \cap \bQ$ or $(a, \infty) \cap \bQ$, because $L$ is ample. 

The proof of \autoref{thm:effective_general_fiber_implies_inf_is_zero} follows the ideas of the proof of \autoref{prop:weak_positivity}. It is centered around objects called Cartier modules \cite[Lemma 13.1]{Gabber_Notes_on_some_t_structures}. These are connected to $\sD$-modules in positive characteristic  \cite{Lyubeznik_F_modules_applications_to_local_cohomology_and_D_modules_in_characteristic_p_greater_than_0,Blickle_The_D_module_structure_of_R_F_modules}. More precisely special Cartier modules, the unit Cartier modules \cite[Definition 2.2]{Blickle_The_D_module_structure_of_R_F_modules}, yield $\sD$-modules. On the other hand the Cartier modules  used in the present article are not unit Cartier modules but are still closely related to  $\sD$-modules (in fact pushforward $\sD$-modules).  This can serve as an intuitive explanation for their appearance, since most of the characteristic zero results in the topic use similar $\sD$-modules originating from Variations of Hodge structures.

To prove \autoref{thm:effective_general_fiber_implies_inf_is_zero} we start with some big enough $s$ such that $\kappa(K_{X/Y} + s f^*L) \geq0$ (see \autoref{lem:generic_Kodaira_dimension_behavior_by_an_ample_twist_from_base}) and we would like to show that we can reduce $s$. That is, we want to  exhibit a sequence $s_m$ of rational numbers such that  $\kappa(K_{X/Y} + s_m f^*L) \geq0$ and $\lim_m s_m = 0$. Without giving every detail here, since the argument is not long (see \autoref{sec:argument}), we find such sequence by using Cartier modules on the source $Y^n$ of the $n$-th iterated absolute Frobenius of $Y$.  The important fact about Cartier modules used at this point  is that they posses similar global generation properties as the famous global generation statement of Mumford through Castelnuovo-Mumford regularity \cite[Theorem 1.8.5]{Lazarsfeld_Positivity_in_algebraic_geometry_I}. Since our Cartier modules are going to be subsheaves of $f_{Y^n,*} \omega_{X_{Y^n}/Y^n}^m \otimes G$ for some line 
bundle $G$, this will 
yield  global generation of a subsheaf of $f_{Y^n,*} \omega_{X_{Y^n}/Y^n}^m \otimes G$.  Here $G$ is an ample enough line bundle on $Y$, where ``ample enough'' depends only on $Y$. So, after proving that the corresponding Cartier modules are not-zero, we obtain non-zero sections of line bundles of the form $\omega_{X_{Y^n}/Y^n}^m \otimes f_{Y^n}^* G$. Then we can move these sections back to $X$, i.e., to sections of $\omega_{X/Y}^m \otimes f^* G'$, where $G= F_Y^{n,*} G''$ and $G'$ is a slightly more ample line bundle than $G''$. By carefully arranging the argument $G'$ will be isomorphic to $L^{c_m}$ for some integer $c_m >0$ and further $s_m=\frac{c_m}{m}$ will converge to zero.  We need the assumption $S^0(X_\eta, \omega_{X_\eta})\neq 0$ to prove that the appearing Cartier modules are not zero.

\subsection{Organization}

In \autoref{sec:Cartier_modules} we introduce Cartier modules, and show the above mentioned global generation statement in \autoref{thm:Cartier_module_globally_generated}. We have to note that we do not claim any credit for this theorem (see the explanation before it). In \autoref{sec:relative_Cartier_modules} we show how  the previously mentioned subspace $S^0$ of Frobenius stable sections behaves in families. An important consequence is that if it is not zero on the generic fiber then certain Cartier modules are not-zero, which will be used later (in \autoref{sec:argument}) as we have already explained. In  \autoref{sec:auxilliary_lemmas} we prove some easier auxilliary lemmas used later in the main argument, some of which might be known for experts. \autoref{sec:argument} contains the main argument. We give the proofs of \autoref{thm:effective_general_fiber_implies_inf_is_zero} and \autoref{cor:subadditivity_of_Kodaira_dimension} there. \autoref{sec:weak_positivity} is the proof of the weak-positivity of 
$S^0 f_{Y^{n},*} \omega_{X_{Y^{n}}/Y^{n}}$ for every $n \gg 0$, which is independent of the rest of the article and is included because of philosophical reasons as explained above. In the appendices we include a  statement that is well-known to experts, but for which we have not found adequate reference at this point.

\subsection{Notation}
\label{sec:notation}

 We fix a perfect base-field $k$ of characteristic $p>0$.
A variety is an integral, separated scheme of finite type over a field (not necessarily $k$ and not necessarily perfect). On the definitions of the Kodaira-Iitaka dimension of a line bundle and the Kodaira dimension of a variety, we refer to \autoref{sec:Kodaira_Iitaka_dimension} and \autoref{appendix:Kodaira_dimension_normalization}. In fact, the notion of Kodaira dimension is not used at all in the article, apart from the already explained implication of \autoref{thm:subadditivity_of_Kodaira_dimension} by \autoref{cor:subadditivity_of_Kodaira_dimension}. Hence, \autoref{sec:Kodaira_Iitaka_dimension} is more important for the argument of the paper, where the definition of Kodaira-Iitaka dimension can be found in \autoref{defn:Kodaira_Iitaka}.  A \emph{fibration} is a surjective morphism  $f: X \to Y$ between  projective varieties  over $k$ such that $f$ is separable and $f_* \sO_X \cong \sO_Y$. Note that by \cite[Theorem 7.1]{Badescu_Algebraic_surfaces} the generic fiber is geometrically integral. We do not 
assume $X$ to be normal, because we want the notion of fibration to be stable under pulling back by the absolute Frobenius morphism of the base. 
By abuse of notation we denote line bundles and any corresponding Cartier divisors by the same letter. We hope that from the context it is always clear which one we mean. For a scheme $X$ of positive characteristic $F_X : X \to X$ (or just simply $F$) denotes the absolute Frobenius morphism. Sometimes the source of $F_X^n$  is denoted by $X^n$, and in these situations if $X$ had a $k$ structure $\nu : X \to \Spec k$, the $k$-structure on $X^n$ will be given by $F_{\Spec k}^n \circ \nu$. Then  $F_X^n$ as a morphism $X^n \to X$ becomes a $k$-morphism. The \emph{generic fiber} is the fiber over the generic point.

We use Cartier divisors on non-normal (but integral) schemes at plenty of places. In these situations we do mean the original definition of Cartier divisors \cite[p. 141]{Hartshorne_Algebraic_geometry}, not the Weil divisor defined by it, since the latter does not make sense at the singular codimension one points. Note that a Cartier divisor $D$ is effective if $1 \in \sL(D)$, where $\sL(D) \subseteq \sK(X)$ is the line bundle corresponding to $D$ as introduced on \cite[p. 144]{Hartshorne_Algebraic_geometry}. Let  $\Ca(X)$ be the group of Cartier divisors on $X$. This is not the class group, so equality in $\Ca(X)$ means actual equality. Then the group of $\bQ$-Cartier divisors is $\bQ$-$\Ca(X)=\Ca(X) \otimes_{\bZ} \bQ$. 

For a proper variety $X$ over $k$ with structure morphism $\nu : X \to \Spec k$, the canonical divisor is $\omega_X := \sH^{-\dim X}(\omega_X^\bullet) := \sH^{-\dim X}(\nu^! \sO_{\Spec k})$ as defined in \cite{Hartshorne_Residues_and_duality}. The canonical divisor $K_X$ is any representative Cartier divisor if $\omega_X$ is a line bundle or any representative Weil divisor if $X$ is normal. When it does not cause any misunderstanding, pullback is denoted by lower index. E.g., if $\sF$ is a sheaf on $X$, and $X \to Y$ and $Z \to Y$ are morphisms, then $\sF_Z$ is the pullback of $\sF$ to $X \times_Y Z$. If $\sF$ is a coherent sheaf on a scheme then $\sF^{[m]}:= \left(\sF^{\otimes m}\right)^{**}$ is the $m$-th reflexive power for any integer $m \geq 0$. These are the sheaves that obey Hartog's theorem on an $S_2$, $G_1$ scheme, i.e., they extend uniquely from an open set obtained by deleting a closed subset of codimension at least two \cite{Hartshorne_Generalized_divisors_on_Gorenstein_schemes}.

\subsection{Acknowledgement}

I am grateful  to Karl Schwede for reading a very preliminary version of the article, and to Christopher Hacon for comments on another slightly less preliminary version. I am also grateful to Yifei Chen for sending comments on the article. 

\section{Cartier modules}
\label{sec:Cartier_modules}

In the present article we use only a special case of what is defined to be a Cartier module in \cite[Definition 8.0.1]{Blickle_Schwede_p_to_the_minus_one_linear_maps_in_algebra_and_geometry}. Note that the original definition of Cartier modules was given earlier in \cite{Blickle_Bockle_Cartier_modules_finiteness_results}. 

\begin{definition} \cite[Definition 8.0.1]{Blickle_Schwede_p_to_the_minus_one_linear_maps_in_algebra_and_geometry}
A Cartier module over a scheme $X$ of positive characteristic is a triple $(M,\tau,g)$, where $M$ is a coherent sheaf of $\sO_X$-modules, $g>0$ an integer and $\tau : F^g_* M \to M$ is an $\sO_X$-linear homomorphism. In this setting $\tau^e : F^{ge}_* M \to M$ is defined as the composition of the following homomorphisms.
\begin{equation*}
\xymatrix@C=50pt{
F^{ge}_* M \ar[r]^{F^{g(e-1)}_* \tau } & F^{g(e-1)}_* M \ar[r]^{F^{g(e-2)}_* \tau} &  \dots \ar[r]^{F^{g}_* \tau} & F^g_* M \ar[r]^\tau & M 
}
\end{equation*}
\end{definition}

\begin{proposition} \cite[Lemma 13.1]{Gabber_Notes_on_some_t_structures} \cite[Proposition 8.1.4]{Blickle_Schwede_p_to_the_minus_one_linear_maps_in_algebra_and_geometry}
\label{prop:Cartier_module_stabilizes}
If  $(M,\tau,g)$ is a Cartier module on a scheme $X$ essentially of finite type over  $k$, then the descending chain $M \supseteq \im \tau \supseteq \im \tau^2 \supseteq \dots $ stabilizes.
\end{proposition}

\begin{definition} 
The stable image of \autoref{prop:Cartier_module_stabilizes} is denoted by $S^0 M$. 
\end{definition}

Note that $\sigma(M)$ is also used in the literature instead of $S^0 M$. We use $S^0 M$ since it aligns with the notation $S^0 f_*(\sigma(X, \Delta) \otimes L)$ introduced in \cite[Definition 2.14]{Hacon_Xu_On_the_three_dimensional_minimal_model_program_in_positive_characteristic}, which will be the only example of Cartier modules used in the present article. 

\begin{lemma} \cite[Proposition 8.1.4]{Blickle_Schwede_p_to_the_minus_one_linear_maps_in_algebra_and_geometry}
\label{lem:Cartier_module_stable_surjective}
If  $(M,\tau,g)$ is a Cartier module on a scheme $X$ essentially of finite type over  $k$, then the restriction of $\tau$ to $S^0M$ induces a surjective homomorphism $\tau : F^g_* S^0 M \to S^0 M$. In particular, $(S^0M, \tau|_{S^0 M }, g)$ is a Cartier module with surjective structure homomorphism. 
\end{lemma}

We do not claim any originality of the following theorem. The method is from \cite{Keeler_Fujita_s_conjecture_and_Frobenius_amplitude} (revised in \cite{Schwede_A_canonical_linear_system}) and the fact that it works also for Cartier modules was communicated by Karl Schwede in a personal conversation.

\begin{theorem} (c.f., \cite[Theorem]{Schwede_A_canonical_linear_system})
\label{thm:Cartier_module_globally_generated}
If $(M, \tau , g)$ is a coherent Cartier module on a projective  scheme $X$ of dimension $n$ over $k$, $A$ is an ample globally generated line bundle and $H$ is an ample line bundle on $X$, then $S^0 M \otimes A^n \otimes H$ is globally generated. 
\end{theorem}

\begin{proof}
By \autoref{prop:Cartier_module_stabilizes},  $\tau^e : F^{ge}_*  M \to S^0M$ is surjective  for every integer $e  \gg 0$. Therefore it is enough to prove that $F^{ge}_* M \otimes A^n \otimes H$ is globally generated for every $e \gg 0$. Hence, by \cite[Theorem 1.8.5]{Lazarsfeld_Positivity_in_algebraic_geometry_I} it is enough to prove that for every $e \gg 0$ and $i>0$,
\begin{equation*}
H^i(X,F^{ge}_* M  \otimes A^{n-i} \otimes H) = 0 .
\end{equation*}
On the other hand since $X$ is $F$-finite, $F$ is a finite morphism, and hence
\begin{equation*}
H^i(X,F^{ge}_* M \otimes A^{n-i} \otimes H) \cong H^i(X, M  \otimes F^{ge,*} A^{n-i} \otimes  F^{ge,*}H)  \cong H^i \left(X, M  \otimes  (A^{n-i} \otimes  H)^{p^{ge}} \right).
\end{equation*}
Since $n -i \geq 0$ and both $A$ and $H$ are ample, Serre-vanishing concludes our proof.
\end{proof}

For the next lemma recall the notion $S^0 f_* (\sigma (X,\Delta) \otimes M)$ introduced in \cite[Definition 2.14]{Hacon_Xu_On_the_three_dimensional_minimal_model_program_in_positive_characteristic}.

\begin{lemma}
\label{lem:we_indeed_define_a_Cartier_module}
Given  a proper morphism of varieties $f : X \to Y$ over $k$ such that $\omega_X$ is a line bundle, a line bundle $M$ on $X$, an integer $g>0$ and an effective $\bQ$-Cartier divisor $\Delta$, such that $(p^g -1) (K_X+ \Delta)$ is a $(\bZ-)$Cartier divisor  and $(p^g -1) (K_X+ \Delta) \sim (p^g -1) M$, the chain
\begin{multline}
\label{eq:we_indeed_define_a_Cartier_module:one_step}
\xymatrix{
\dots \ar[r] & 
f_* F^{g(e+1)}_* \sO_X \left(  p^{g(e+1)} M  + \left(1-p^{g(e+1)} \right) (K_X+ \Delta) \right)  
\ar[r]&}
\\
\xymatrix{\ar[r] &  
f_* F^{ge}_* \sO_X ( p^{ge} M - (1-p^{ge}) (K_X+ \Delta) ) 
\ar[r] & \dots
}
\end{multline}
known from the theory of $F$-singularities is isomorphic to 
\begin{equation*}
\xymatrix@C=45pt{
\dots \ar[r] & 
  F^{g(e+1)}_* f_* M
\ar[r]^{ F^{ge}_* f_*(\alpha)} &
 F^{ge}_* f_*  M
\ar[r]^{ F^{g(e-1)}_* f_* (\alpha)}  &
 F^{g(e-1)}_* f_* M
\ar[r] & \dots,
}
\end{equation*}
where $\alpha$ is the following homomorphism induced by the Grothendieck trace of Frobenius.
\begin{equation*}
F_*^g M
 \cong F_*^g \sO_{X}( p^g M +  (1-p^g) (K_X + \Delta))
  \cong M \otimes F_*^g \sO_{X}((1-p^g) (K_X+ \Delta))
   \to M 
\end{equation*}
In particular, $(f_* M, f_* (\alpha), g)$ and $\left( S^0 f_* (\sigma (X,\Delta) \otimes M), f_*(\alpha)|_{S^0 f_* (\sigma (X,\Delta) \otimes M)}, g \right)$ are Cartier modules.


\end{lemma}

\begin{proof}
We claim that \eqref{eq:we_indeed_define_a_Cartier_module:one_step} is isomorphic to 
\begin{equation}
\label{eq:we_indeed_define_a_Cartier_module:stream}
\xymatrix@C=45pt{
\dots \ar[r] & 
 f_* F^{g(e+1)}_* M
\ar[r]^{f_* F^{ge}_*(\alpha)} &
f_* F^{ge}_*  M
\ar[r]^{f_* F^{g(e-1)}_*(\alpha)}  &
f_* F^{g(e-1)}_*  M
\ar[r] & \dots,
}.
\end{equation}
Then the statement of the lemma follows since $F_* f_* = f_* F_*$. 

To show our claim, note that the map of \eqref{eq:we_indeed_define_a_Cartier_module:one_step}, using projection formula for $F$, is
\begin{multline*}
f_*  F_*^{ge}  \Big(  \sO_{X} \big(  p^{ge} M + (1-p^{ge}) (K_X + \Delta) \big)  \otimes F_*^g \sO_{X} \big( (1-p^g) (K_X + \Delta)\big) \Big)   \to \\
\to f_*  F_*^{ge}    \sO_{X} \big(  p^{ge} M + (1-p^{ge}) (K_X + \Delta) \big) ,
\end{multline*}
where Frobenius trace is applied to $F^g_* \sO_X((1-p)(K_X + \Delta) )$ and everything else is sent via identity. However, this map agrees with $f_* F^{ge}_* (\alpha)$ of \eqref{eq:we_indeed_define_a_Cartier_module:stream}, since 
\begin{equation*}
p^{ge} M + (1-p^{ge}) (K_X + \Delta) \sim M. 
\end{equation*}
\end{proof}

\begin{remark}
In \autoref{lem:we_indeed_define_a_Cartier_module}, if we assumed $X$ to be   $S_2$ and $G_1$, then we could  assume $\Delta$ to be a $\bQ$-divisorial sheaf \cite[Section 2.1]{Miller_Schwede_Semi_log_canonical}, or if $X$ was normal, then $\Delta$ could be a $\bQ$-Weil-divisor. The statement and the proof would be verbatim the same. On the other hand, in the present article we only use the case of a $\bQ$-Cartier divisor $\Delta$ as stated in \autoref{lem:we_indeed_define_a_Cartier_module}.
\end{remark}

\section{Behavior of relative Cartier modules}
\label{sec:relative_Cartier_modules}

In \cite[Definition 4.1]{Schwede_A_canonical_linear_system} a subgroup of $S^0(X, L) \subseteq H^0(X, L)$ was introduced, where $X$ is an arbitrary scheme of finite type over $k$ and $L$ a line bundle on $X$. One of the fundamental usages of this subgroup is  that its elements lift from sharply $F$-pure centers \cite[Prop 5.3]{Schwede_A_canonical_linear_system}, which had many applications recently in higher dimensional geometry (e.g., \cite{Hacon_Xu_On_the_three_dimensional_minimal_model_program_in_positive_characteristic,Patakfalvi_Semi_positivity_in_positive_characteristics}). A natural question to ask is then how does this subgroup behave in families. A theory concerning this was worked out in  \cite{Patakfalvi_Schwede_Zhang_F_singularities_in_families}.  However, to apply it in the setting of the current article it needs some modifications, since there relative ampleness on the fibers was used in an essential way. Luckily, if one concerns the special case of  $S^0(F,\sigma(F,\Delta_F) \otimes M_F)$, 
where $S^0 f_* (\sigma(X,\Delta) \otimes M)$ is a Cartier module as in \autoref{lem:we_indeed_define_a_Cartier_module}, it turns out that the relative ampleness is not needed. Our setup is the following.

\begin{notation}
\label{notation:openness_S_0}
Let $f : X \to Y$ be a proper, surjective morphism of varieties over $k$ with  $Y$ regular. Let $\xi$  be the generic  point of $Y$. Assume that $\Delta$ is a $\bQ$-Cartier divisor, $M$ is a line bundle on $X$ and $g>0$ an integer, such that 
\begin{enumerate}
\item $X_\xi$ is geometrically integral,
\item $X_\xi$ is $S_2$ and $G_1$,
\item $\Delta_\xi$ is  effective, 
\item \label{itm:openness_S_0:linear_equivalence} $(1-p^g)(K_{X_\xi} +   \Delta_\xi) \sim  (1-p^g) M_\xi$.
\end{enumerate}
\end{notation}

\begin{remark}
Similarly to \autoref{lem:we_indeed_define_a_Cartier_module}, in \autoref{notation:openness_S_0} if  we assumed $X$ to be $S_2$, $G_1$ as well then we could  assume $\Delta$ to be a   $\bQ$-Weil-divisor that avoids the codimension one singular points. Similarly, we could assume both $X$ and $X_\eta$ to be normal, and then $\Delta$ could be just any $\bQ$-Weil-divisor. The statements of the section and the proofs would be verbatim the same. On the other hand, in the present article we only use the generality stated in \autoref{notation:openness_S_0}.
\end{remark}

Our main goal  in this section is the following. Recall that a perfect point of $W$ is  a morphism $\Spec k' \to W$ where $k'$ is a perfect field. 

\begin{theorem}
\label{thm:openness_S_0}
In the situation of \autoref{notation:openness_S_0}, there is a Zariski open set $W$ of $Y$ such that $S^0(F,\sigma(F, \Delta_F) \otimes M_F)$ has the same dimension for every fiber $F$ over every perfect point of $W$. Further, the rank of $S^0 f_* \left(\sigma(X_W, \Delta|_W) \otimes M|_W \right)$ is at least as big as this general value.
\end{theorem}

In the above statement is implicitly included that $\Delta_F$ is meaningful for every fiber $W$ over every perfect point of $W$. 
The following lemma states this and some other similar reductions. The proof is immediate.

\begin{lemma}
\label{lem:openness_S_0_assumptions}
In the situation of \autoref{notation:openness_S_0}, there is a non-empty open set $U \subseteq Y$, such that 
\begin{enumerate}
\item $U$ is affine, 
\item $f|_{f^{-1}(U)}$ has geometrically integral fibers,
\item $f|_{f^{-1}(U)}$ is flat and relatively $S_2$ and $G_1$,
\item $\Delta|_{f^{-1}(U)}$ is  effective and contains no fiber,
\item \label{itm:openness_S_0_assumptions:linear_equivalence} $(1-p^g)\left(K_{f^{-1}(U)/U} +  \Delta|_{f^{-1}(U)} \right) \sim (1-p^g) M|_{f^{-1}(U)}$.
\end{enumerate}
\end{lemma}

\begin{proof}
Use \cite[Thm 6.9.1]{Grothendieck_Elements_de_geometrie_algebrique_IV_II} and 
\cite[Thm 12.2.4]{Grothendieck_Elements_de_geometrie_algebrique_IV_III} for flatness, geometric integrality and the $S_2$ property. Since $Y$ is regular, a fiber is Gorenstein at a point $P$ if and only if so is the total space at $P$. Now use the openness of the Gorenstein locus to prove the $G_1$ property. To prove \autoref{itm:openness_S_0_assumptions:linear_equivalence} of \autoref{lem:openness_S_0_assumptions}, note that by  \autoref{itm:openness_S_0:linear_equivalence} of \autoref{thm:openness_S_0} and cohomology and base-change, $f_* \sO_X((1-p^g)(K_{X/Y} +   \Delta - M))$ has rank 1, and hence one just have to choose a $U$ contained in the locally free locus of $f_* \sO_X((1-p^g)(K_{X/Y} +   \Delta - M))$. 
\end{proof}

Since our aim is to prove \autoref{thm:openness_S_0}, by \autoref{lem:openness_S_0_assumptions}, after replacing $Y$ by one of its open sets, we may assume the properties listed in \autoref{lem:openness_S_0_assumptions} for $U=Y$. 
In particular, then $ f : X \to Y$ and $\Delta$ almost satisfy \cite[Notation 2.1, Definition 2.11 and Notation 6.1]{Patakfalvi_Schwede_Zhang_F_singularities_in_families} except the following two conditions:

 First, in \cite[Notation 6.1]{Patakfalvi_Schwede_Zhang_F_singularities_in_families} projectivity is assumed. However, really only properness is used in the proofs referenced here. Hence this should not cause any problem.

 More importantly, in \cite[Definition 2.11]{Patakfalvi_Schwede_Zhang_F_singularities_in_families} it is assumed that $\Delta$ avoids the codimension one points of the fibers that are not in the smooth locus of $f$. We do not assume this for our $\Delta$. However, we do assume that $\Delta$ is $\bQ$-Cartier, which only means in our case that $\Delta$ is formal sum of Cartier divisors with rational coefficients (see \autoref{sec:notation}). 

Luckily the setting of \cite{Patakfalvi_Schwede_Zhang_F_singularities_in_families} is really more general than that of \cite[Definition 2.11]{Patakfalvi_Schwede_Zhang_F_singularities_in_families}. The latter definition is just one way to obtain a line bundle $L$ and a homomorphism $\varphi : L^{1/p^g} \to R_{A^{1/p^g}}$ as in \cite[Definition 2.6]{Patakfalvi_Schwede_Zhang_F_singularities_in_families} (note that in \cite{Patakfalvi_Schwede_Zhang_F_singularities_in_families}, $e$ is used instead of $g$). The construction of \cite[Remark 2.10]{Patakfalvi_Schwede_Zhang_F_singularities_in_families} yields a $\varphi$ for $\bQ$-Cartier $\Delta$ even if $\Delta$ goes through non-smooth codimension one points of the fibers. One just uses the construction of \cite[Remark 2.10]{Patakfalvi_Schwede_Zhang_F_singularities_in_families}  and disregards condition $(c)$ that $D$ avoids all the non-smooth codimension one points of the fibers. Denote by $\varphi_{\Delta}$ the obtained homomorphism. Of course this way one looses the equivalence of the different choices of the effective $\bQ$-Cartier divisors $\Delta$ and of the 
relatively divisorial homomorphisms $\varphi$ as explained in \cite[Remark 2.10]{Patakfalvi_Schwede_Zhang_F_singularities_in_families}. However, this is equivalence is not necessary to apply the results of \cite{Patakfalvi_Schwede_Zhang_F_singularities_in_families} to the above obtained homomorphism  $\varphi_\Delta : L^{1/p^g} \to R_{A^{1/p^g}}$.

On the other hand to translate the results obtained for $\varphi_\Delta$ back to the divisorial language, we do have to check the compatibility of this $\varphi_\Delta$ construction with base-change, and also the independence  of $\varphi_\Delta$ from the choice of $g$. These are worked out in \cite[Lemma 2.14, Lemma 2.20 and Lemma 2.23]{Patakfalvi_Schwede_Zhang_F_singularities_in_families}. Again the statements are unfortunately worded slightly inadequately for our setup. So, for example in \cite[Lemma 2.14]{Patakfalvi_Schwede_Zhang_F_singularities_in_families}, instead of proving that $\Delta_\varphi= \Delta_\varphi^n$, we should prove that $\varphi_\Delta$ becomes $(\varphi_\Delta)^n$ if we multiply the $g$ that we are using in the construction of $\varphi_\Delta$ by $n$ (simply because as explained above $\Delta_\varphi$ does not make sense in our setup). On the other hand the proof works with almost no change. In \cite[Lemma 2.20]{Patakfalvi_Schwede_Zhang_F_singularities_in_families} the correct 
statement is that $(\varphi_{\Delta})_T= \varphi_{\Delta_T}$ and in \cite[Lemma 2.23]{Patakfalvi_Schwede_Zhang_F_singularities_in_families} it is 
that $\varphi_{\Delta}|_{X_s}$ can be identified with the usual map $\psi_{\Delta_s} : L_s^{1/p^g} \to R_s$ given by $\Delta_s$. The proofs again work with almost no change. In fact the latter one is basically immediate. For the former one, i.e., for $(\varphi_{\Delta})_T= \varphi_{\Delta_T}$, there is one non-trivial input that the Grothendieck trace of the relative Frobenius morphism is compatible with base-change \cite[Lemma 2.17]{Patakfalvi_Schwede_Zhang_F_singularities_in_families}. 

Hence we see that the $\Delta \mapsto \varphi_\Delta$ is indeed compatible with base-change if $\Delta$ is a $\bQ$-Cartier divisor, even if we allow $\Delta$ to go through singular codimension one points of the fibers. Therefore we may use every result of \cite{Patakfalvi_Schwede_Zhang_F_singularities_in_families}, which uses only the homomorphism (i.e., the $\varphi_\Delta$) language. This is true for almost all statements of that article. Unfortunately we are also going to use one of the few statements \cite[Proposition 6.33]{Patakfalvi_Schwede_Zhang_F_singularities_in_families} that is not phrased using the homomorphism language. However, that proof works verbatim in our situation.

Having discussed the necessary adjustments one has to do to \cite{Patakfalvi_Schwede_Zhang_F_singularities_in_families} to apply it in our case, let us start with the setup for proving \autoref{thm:openness_S_0}.  In the situation of \autoref{notation:openness_S_0} with the assumptions made after \autoref{lem:openness_S_0_assumptions}, consider the following sheaf introduced in  \cite[Definition 6.4]{Patakfalvi_Schwede_Zhang_F_singularities_in_families} (there $e$ is used instead of $g$ and $n$ instead of $e$).
\begin{equation}
\label{eq:S_0_definition}
 S_{\Delta,ge} f_* M = \im \left( f_{Y^{ge},*} F_{X^{ge}/Y^{ge},*} \sO_X \left( \left( 1-p^{ge} \right)(K_{X/Y} + \Delta) + p^{ge} M\right)  \to  f_{Y^{ge},*} M_{Y^{ge}} \right)
\end{equation}
where:
\begin{itemize}
\item $X^{ge}$ and $Y^{ge}$ are the source spaces of the $ge$-times iterated absolute Frobenius morphism of $X$ and $Y$, respectively. In particular, $X^{ge}$ and $Y^{ge}$ are identical to  $X$ and $Y$ as schemes. However the $k$ structure on $X^{ge}$ and $Y^{ge}$ is given by the respective $k$-structures twisted  via the map $k \to k$ given by $x \mapsto x^{p^{ge}}$.
\item Writing $\sO_X \left( \left( 1-p^{ge} \right)(K_{X/Y} + \Delta) + p^{ge} M\right)$ is a slight abuse of notation. We should really write:
\begin{equation*}
\sO_{X^{ge}} \left( \left( 1-p^{ge} \right)(K_{X^{ge}/Y^{ge}} + \Delta_{ge}) + p^{ge} M_{ge}\right),
\end{equation*}
where $\Delta_{ge}$ and $M_{ge}$ are just the divisors that become  $\Delta$ and $M$ when we forget the $k$ structure of $X^{ge}$ and regard it as an abstract scheme identical to $X$. 
\item $F_{X^{ge}/Y^{ge}} : X^{ge} \to X_{Y^{ge}}$ is the relative Frobenius morphism, which fits into the following commutative diagram,
\begin{equation*}
\xymatrix{
X^{ge} \ar@/_1pc/[ddr]_{f^{ge}}   \ar[dr]|{F_{X^{ge}/Y^{ge}}} \ar@/^1pc/[drr]^{F_X^{ge}} \\
& X_{Y^{ge}} \ar[r]_{\left(F^{ge}_Y\right)_X} \ar[d]^{f_{Y^{ge}}} & X \ar[d]^f \\ 
& Y^{ge} \ar[r]_{F^{ge}_Y} & Y
}
\end{equation*}
\item We have
\begin{equation*}
\qquad \qquad F_{X^{ge}/Y^{ge},*} \sO_X \left( \left( 1-p^{ge} \right)(K_{X/Y} + \Delta) + p^{ge}M \right) 
\\ \cong
M_{Y^{ge}} \otimes F_{X^{ge}/Y^{ge},*}  \sO_X \left( \left( 1-p^{ge} \right)(K_{X/Y} + \Delta) \right)  
\end{equation*}
So, the homomorphism in \autoref{eq:S_0_definition}, is given by tensoring $\id_{M_{Y^{ge}}}$ with the trace homomorphism  obtained from Grothendieck duality by tracing the image of the section $s_\Delta$ corresponding to $(p^{ge} -1)\Delta$  through the following stream of isomorphisms. 
\begin{multline*}
\qquad \qquad \Hom_{X_{Y^{ge}}} \left( F_{X^{ge}/Y^{ge},*} \sO_X \left( \left( 1-p^{ge} \right)(K_{X/Y} + \Delta) \right), \sO_{X_{Y^{ge}}} \right) 
\\ %
\begin{split}
& \cong \Hom_{X} \left( \sO_X \left( \left( 1-p^{ge} \right)(K_{X/Y} + \Delta) \right), F_{X^{ge}/Y^{ge}}^! \sO_{X_{Y^{ge}}} \right) 
\\ & \cong \Hom_{X} \left( \sO_X \left( \left( 1-p^{ge} \right)(K_{X/Y} + \Delta) \right),  \sO_X \left( \left( 1-p^{ge} \right)K_{X/Y}  \right) \right) 
\\ & \cong H^0(X, (p^{ge} -1) \Delta) \ni  s_\Delta 
\end{split}.
\end{multline*}
Note that this also equals the composition of the natural embedding
\begin{equation*}
F_{X^{ge}/Y^{ge},*}  \sO_X \left( \left( 1-p^{ge} \right)(K_{X/Y} + \Delta) \right)   \to F_{X^{ge}/Y^{ge},*}  \sO_X \left( \left( 1-p^{ge} \right)K_{X/Y}  \right)  
\end{equation*}
composed with the trace 
\begin{equation*}
F_{X^{ge}/Y^{ge},*}  \sO_X \left( \left( 1-p^{ge} \right)K_{X/Y}  \right)  \to \sO_{X_{Y^{ge}}}
\end{equation*}
of the relative Forbenius $F_{X^{ge}/Y^{ge}}$.
\item Note that $S_{\Delta, ge} f_* M$ contrary to the notation is not a sheaf on $Y$, but it is a sheaf on $Y^{ge}$. Furthermore, it is a subsheaf of $f_{Y^{ge},*} M_{Y^{ge}} \cong F_Y^{ge,*} f_* M$ (recall, $Y$ is regular, so $F_Y$ is flat).
\item Further note that 
\begin{equation*}
\left( 1-p^{ge} \right)(K_{X/Y} + \Delta) + p^{ge}M \sim M ,
\end{equation*}
so in \autoref{eq:S_0_definition} we really have a map $f_{Y^{ge},*} F_{X^{ge}/Y^{ge},*} M  \to  f_{Y^{ge},*} M_{Y^{ge}}$.
\end{itemize}
Also note that \cite{Patakfalvi_Schwede_Zhang_F_singularities_in_families} uses a slightly different notation which is more suitable for dealing with relative Frobenii. In fact, we will also be forced to use that notation later in the section. Matching up the two notations, i.e., to see that \autoref{eq:S_0_definition} is equivalent to \cite[Definition 6.4]{Patakfalvi_Schwede_Zhang_F_singularities_in_families} should be just a matter of familiarity with both notations.

Now, if $y \in Y$ is a perfect point (i.e. a map $\Spec K \to Y$ for some perfect field $K$), then there is a base-change diagram \cite[equation (6.16.4)]{Patakfalvi_Schwede_Zhang_F_singularities_in_families}
\begin{equation}
\label{eq:base_change_S_0}
\xymatrix{
f_{Y^{ge},*} F_{X^{ge}/Y^{ge},*} M \otimes_{\sO_{Y^{ge}}} k(y^{ge}) \ar[d] \ar[r] &   f_{Y^{ge},*} M_{Y^{ge}} \otimes_{\sO_{Y^{ge}}} k(y^{ge}) \ar[d] \\
 H^0(X_y, F_{X_y,*}^{ge} M_y ) \ar[r] & H^0(X_y,  M_y) 
},
\end{equation}
where $y^{ge}$ is the $p^{ge}$ twisted version of $y$, which is isomorphic to $y$ by the perfectness assumption. (Note that the vertical arrows are just the usual base-change maps from cohomology and base-change.) Here the image of $f_{Y^{ge},*} F_{X^{ge}/Y^{ge},*} M \to   f_{Y^{ge},*} M_{Y^{ge}}$ is $S_{\Delta, ge} f_* M$, while of the bottom horizontal arrow is $S^0(X_y, \sigma(X,\Delta) \otimes M_y)$ for $e \gg 0$. Let us denote by $S_{\Delta,ge} (X_y, M_y)$ the image of the latter for a given $e$. Note that diagram \autoref{eq:base_change_S_0} is the very reason for the introduction of $S_{\Delta,ge} f_* M$. It is the only known sheaf that base-changes to $S^0(X_y, \sigma(X,\Delta) \otimes M_y)$ in some sense.

Let $U \subseteq Y$ be the (non-empty) open set where $H^0(X_y,M_y)$ is constant. Now, notice that 
\begin{equation*}
H^0(X_y, M_y) \cong H^0 \left(X_y, F_{X_y,*}^{ge} M_y \right) \cong H^0 \left(X_{y^{ge}}, M_{y^{ge}} \right) \cong H^0\left(X_{y^{ge}}, \left(M_{Y^{ge}}\right)_{y^{ge}}\right).
\end{equation*}
Furthermore, both $f_{Y^{ge}}$ and $f_{Y^{ge}} \circ F_{X^{ge}/Y^{ge}} : X^{ge} \to Y^{ge}$ are flat. Hence for any perfect point $y \in U$, the vertical base-change maps of \eqref{eq:base_change_S_0} are isomorphisms \cite[Corollary III.12.9]{Hartshorne_Algebraic_geometry}. In particular, for every perfect point $y \in U$, the natural map
\begin{equation}
\label{eq:base_change_S_0_result}
S_{\Delta, ge} f_* M \otimes_{Y^{ge}} k(y^{ge}) \to S_{\Delta, ge} (X_y, M_y) .
\end{equation}
is surjective, and further it is an isomorphism, whenever 
\begin{equation}
\label{eq:base_change_S_0_condition}
S_{\Delta, ge} f_* M \otimes_{Y^{ge}} k(y^{ge}) \to f_{Y^{ge},*} M_{Y^{ge}} \otimes_{Y^{ge}} k(y^{ge})
\end{equation}
is injective.
The following proposition is the main ingredient of our discussion.

\begin{proposition}
\label{prop:stabilization}
In the above situation, there is a non-empty Zariski open set $V \subseteq U$, such that for every $e \gg 0$, 
\begin{equation*}
F_{V^{ge}}^{g,*} \left( S_{\Delta,ge} f_* M|_{V^{ge}} \right)  = S_{\Delta, g(e+1)} f_* M |_{V^{g(e+1)}} 
\end{equation*}
as subsheaves of $F_V^{g(e+1),*} f_* M$. 
\end{proposition}

\begin{proof}
Consider the above base-change discussion for $y$ being the prefect closure of the general point of $Y$ (or equivalently of $U$). Then \autoref{eq:base_change_S_0_condition} is isomorphism and hence \autoref{eq:base_change_S_0_result} is an isomorphism as well. Using this and the fact that for every $e \gg 0$
\begin{equation*}
S_{\Delta,ge} (X_y, M_y) = S^0(X_y, M_y),
\end{equation*}
we see that $S_{\Delta,ge} f_* M \otimes_{Y^{ge}} k(y^{ge}) $ is the same for every $e \gg 0$. In other words the rank of $S_{\Delta,ge} f_* M$ stabilizes for $e \gg 0$. Further, by \cite[Proposition 6.6]{Patakfalvi_Schwede_Zhang_F_singularities_in_families}, for any $ e \geq 0$, 
\begin{equation*}
F_{V^{ge}}^{g,*} \left( S_{\Delta,ge} f_* M|_{V^{ge}} \right) \supseteq S_{\Delta,g(e+1)} f_* M|_{V^{g(e+1)}}.
\end{equation*}
as subsheaves of $f_{V^{g(e+1)},*} (M_{V^{g(e+1)}}) \cong F^{g(e+1),*}_V f_* M$.
Hence there is an integer $e>0$, and an open set $V \subseteq U$, such that 
\begin{equation}
\label{eq:stabilization:V}
F_{V^{ge}}^{g,*} \left( S_{\Delta,ge} f_* M|_{V^{ge}} \right) \cong S_{\Delta,g(e+1)} f_* M|_{V^{g(e+1)}}.
\end{equation}
The reason why we are not ready is that $V$ can be different for different values of $e$. We need to show that the $V$ found for a fixed $e=e_0$ works for all $e>e_0$. We prove this by induction on $e$. So, assume that \autoref{eq:stabilization:V} holds for some $e$ and $V$. We are going to prove that it also holds for $e+1$ with the same $V$. However before proceeding we need to change to the notation of \cite{Patakfalvi_Schwede_Zhang_F_singularities_in_families}, since working out every detail of the remaining part of the proof is very tedious without doing so. 

Note that $X^l_{Y^r}$ has the same underlying topological space  as $X$ for every integer $r \geq l \geq 0$. Further, since we assumed that $Y$ is affine, we really have to work only with one topological space $X$, and keep track of the different sheaves of algebras on it. In accordance with \cite[Notation 2.1, Definition 2.6]{Patakfalvi_Schwede_Zhang_F_singularities_in_families} introduce $R:=\sO_X$, $A:=H^0(Y, \sO_Y)$ and $L:=\sO_X((1-p^g)(K_X + \Delta))$. In this notation for example $\sO_{X^{ge}}$ becomes $R^{1/p^{ge}}$ or $\sO_{X_{Y^{ge}}}$ becomes $R \otimes_A A^{1/p^{ge}}$. The usefulness of this notation is apparent for example when one considers isomorphisms of the form
\begin{equation*}
\left(R \otimes_A A^{1/p^{ge}} \right)^{1/p^{ge'}} \cong R^{1/p^{ge'}} \otimes_{A^{1/p^{ge'}}} A^{1/p^{g(e+e')}}.
\end{equation*}
Similarly if $M$ is considered on $X^{ge}$ instead of $X$, which has been denoted by $M_{ge}$ so far, then we write $M^{1/p^{ge}}$ in the rest of the proof.
Further, since the pushforward does not depend on the algebra structures (the algebra structure just influences what module structure we endow the otherwise identical pushforwards with), we can use only $f_*$ instead of the functors of the form $f_{Y^l,*}$. As in \cite{Patakfalvi_Schwede_Zhang_F_singularities_in_families} we denote by $\varphi$ the homomorphism 
\begin{equation*}
F_{X^g/Y^g,*} \sO_X((1-p^g)(K_X + \Delta)) = L^{1/p^g} \to R \otimes_A A^{1/p^g} = \sO_{X_{Y^g}}
\end{equation*}
introduced above. Note that then the other homomorphisms 
\begin{equation*}
F_{X^{ge}/Y^{ge},*} \sO_X((1-p^{ge})(K_X + \Delta)) = \left( L^\frac{p^{ge}-1}{p^g -1} \right)^{1/p^{ge}} \to R \otimes_A A^{1/p^{ge}} = \sO_{X_{Y^{ge}}}
\end{equation*}
which were also introduced above and are denoted by $\varphi^e$ in \cite{Patakfalvi_Schwede_Zhang_F_singularities_in_families} fit into a commutative diagram as follows \cite[2.12, Lemma 2.14, proof of Proposition 3.3]{Patakfalvi_Schwede_Depth_of_F_singularities}.
\begin{equation*}
\xymatrix@R=40pt{
\ar@/_12pc/[dd]|(0.7){ \left( \varphi^{e+1} \otimes_R L \right)^\frac{1}{p^g} } \left( L^\frac{p^{g(e+2)}-1}{p^g -1} \right)^\frac{1}{p^{g(e+2)}}  \ar@/_18pc/[ddd]|(0.8){\varphi^{e+2}} \\
\left( L^\frac{p^{g(e+1)}-1}{p^g -1} \right)^\frac{1}{p^{g(e+1)}} \otimes_{A^\frac{1}{p^{g(e+1)}}} A^\frac{1}{p^{g(e+2)}} 
\ar[d]^{ \left(\varphi^e \otimes_R L \right)^\frac{1}{p^g} \otimes_{A^\frac{1}{p^{g(e+1)}}} A^\frac{1}{p^{g(e+2)}}} 
\ar@/^18pc/[dd]|(0.8){\varphi^{e+1} \otimes_{A^\frac{1}{p^{g(e+1)}}} A^\frac{1}{p^{g(e+2)}}} \\
 L^\frac{1}{p^{g}} \otimes_{A^\frac{1}{p^{g}}} A^\frac{1}{p^{g(e+2)}} \ar[d]_{ \varphi \otimes_{A^\frac{1}{p^g}} A^\frac{1}{p^{g(e+2)}} } \\
R \otimes_A A^\frac{1}{p^{g(e+2)}}  \\
}
\end{equation*}
Apply now $\_ \otimes_R M$ to the above diagram. This yields the following commutative diagram using the projection formula at multiple places.
\begin{equation*}
\xymatrix@R=40pt{
\ar@/_12pc/[dd]|(0.7){ \left( \varphi^{e+1} \otimes_R L \right)^\frac{1}{p^g} \otimes_R M} \left( L^\frac{p^{g(e+2)}-1}{p^g -1} \otimes_R M^{p^{g(e+2)}} \right)^\frac{1}{p^{g(e+2)}}  \ar@/_18pc/[ddd]|(0.8){\varphi^{e+2}\otimes_R M} \\
\left( L^\frac{p^{g(e+1)}-1}{p^g -1} \otimes_R M^{p^{g(e+1)}} \right)^\frac{1}{p^{g(e+1)}} \otimes_{A^\frac{1}{p^{g(e+1)}}} A^\frac{1}{p^{g(e+2)}} 
\ar[d]^{ \left(\varphi^e  \otimes_R L \right)^\frac{1}{p^g} \otimes_{A^\frac{1}{p^{g(e+1)}}} A^\frac{1}{p^{g(e+2)}}\otimes_R M } 
\ar@/^18pc/[dd]|(0.8){\varphi^{e+1} \otimes_{A^\frac{1}{p^{g(e+1)}}} A^\frac{1}{p^{g(e+2)}}\otimes_R M} \\
 \left( L \otimes_R M^{p^g} \right)^\frac{1}{p^{g}} \otimes_{A^\frac{1}{p^{g}}} A^\frac{1}{p^{g(e+2)}}  \ar[d]_{ \varphi \otimes_{A^\frac{1}{p^g}} A^\frac{1}{p^{g(e+2)}} \otimes_R M } \\
M \otimes_A A^\frac{1}{p^{g(e+2)}}  \\
}
\end{equation*}
Now, notice that $M^{1-p^g} \cong L$, so after applying $f_*$ the above diagram becomes:
\begin{equation}
\label{eq:S_0_stabilizes}
\xymatrix@R=40pt{
\ar@/_12pc/[dd]|(0.7){ \alpha:= f_* \left( \left( \varphi^{e+1}  \otimes_R M \right)^\frac{1}{p^g}  \right)}  f_* \left( M^\frac{1}{p^{g(e+2)}} \right)  \ar@/_21pc/[ddd]|(0.8){\beta:=f_* ( \varphi^{e+2} \otimes_R M)} \\
f_* \left( M ^\frac{1}{p^{g(e+1)}} \otimes_{A^\frac{1}{p^{g(e+1)}}} A^\frac{1}{p^{g(e+2)}}  \right)
\ar[d]^{ \gamma:=f_* \left( \left(\varphi^e \otimes_R M \right)^\frac{1}{p^g} \otimes_{A^\frac{1}{p^{g(e+1)}}} A^\frac{1}{p^{g(e+2)}}   \right)} 
\ar@/^23pc/[dd]|(0.8){ \delta:=f_* \left( \varphi^{e+1} \otimes_{A^\frac{1}{p^{g(e+1)}}} A^\frac{1}{p^{g(e+2)}}  \otimes_R M \right)} \\
f_* \left( M^\frac{1}{p^{g}} \otimes_{A^\frac{1}{p^{g}}} A^\frac{1}{p^{g(e+2)}} \right)  \ar[d]_{ \varepsilon :=f_* \left( \varphi \otimes_{A^\frac{1}{p^g}} A^\frac{1}{p^{g(e+2)}}  \otimes_R M \right) } \\
f_* (M \otimes_A A^\frac{1}{p^{g(e+2)}})  \\
}
\end{equation}
Since $F_* f_* = f_* F_*$ and $A \to A^{1/p}$ is flat (because $Y$ is regular), $\alpha$ and $\gamma$ are isomorphic to $ \left( f_*  \left( \varphi^{e+1}  \otimes_R M \right) \right)^\frac{1}{p^g}$ and $ \left( f_* \left(\varphi^e \otimes_R M \right) \otimes_{A^\frac{1}{p^{ge}}} A^\frac{1}{p^{g(e+1)}}   \right)^\frac{1}{p^g}$, respectively.
Hence by the inductional hypothesis, over $V$
\begin{multline}
\label{eq:inductional_consequence}
 \im \alpha = \im \left( f_*  \left( \varphi^{e+1}  \otimes_R M \right) \right)^\frac{1}{p^g} = 
\underbrace{ \left( S_{\Delta,g(e+1)} f_* M \right)^\frac{1}{p^{g}} =  \left( S_{\Delta,ge} f_* M \otimes_{A^\frac{1}{p^{ge}}} A^\frac{1}{p^{g(e+1)}}  \right)^\frac{1}{p^{g}} }_{\textrm{by the inductional hypothesis}}
\\  = \im \left( f_* \left(\varphi^e \otimes_R M \right) \otimes_{A^\frac{1}{p^{ge}}} A^\frac{1}{p^{g(e+1)}}   \right)^\frac{1}{p^g} =  \im \gamma.
\end{multline}
However then,  still over $V$:
\begin{equation*}
S_{\Delta,g(e+1)} f_* M  \otimes_{A^\frac{1}{p^{g(e+1}}} A^\frac{1}{p^{g(e+2)}}
= \underbrace{\im \delta}_{\textrm{by \eqref{eq:S_0_definition}}}  
= \underbrace{\im \varepsilon \circ \gamma}_{\textrm{by \eqref{eq:S_0_stabilizes}}} 
 = \underbrace{\im \varepsilon \circ \alpha}_{\textrm{by \eqref{eq:inductional_consequence}}} 
= \underbrace{\im \beta}_{\textrm{by \eqref{eq:S_0_stabilizes}}} 
= \underbrace{S_{\Delta,g(e+2)} f_* M}_{\textrm{by \eqref{eq:S_0_definition}}} .
\end{equation*}
This concludes our inductional step and hence our proof as well.
\end{proof}

\begin{proof}[Proof of Theorem \ref{thm:openness_S_0}]
First we show the main statement. By \autoref{prop:stabilization}, we may choose a non-empty open subset $W \subseteq V$ over which $S_{ge} f_* (\sigma(X,\Delta) \otimes M)$ and its cokernel in $f_{Y^{ge},*} M_{Y^{ge}}$  are locally free for every (or equivalently one) $e \gg 0$. Hence over $W$ the map of \autoref{eq:base_change_S_0_condition} is isomorphism and hence so is \autoref{eq:base_change_S_0_result}. Then by combining this with Proposition \ref{prop:stabilization} we obtain that for every perfect point $y \in W$, $S_{\Delta,ge}(X_y,  M_y)$ is the same for every $e \gg0$ (where the lower bound on $e$ is independent of $y$). Hence, $S_{\Delta,ge}(X_y, M_y)$ stabilizes at the same value, and therefore $H^0(X_y, \sigma(X_y,\Delta_y) \otimes M_y)$ has the same dimension, for every perfect point $y \in W$.

To see the addendum,  apply \cite[Proposition 6.33]{Patakfalvi_Schwede_Zhang_F_singularities_in_families}.
\end{proof}

\section{Auxilliary lemmas}
\label{sec:auxilliary_lemmas}

\begin{proposition}
\label{prop:base_globally_generated}
Let $Y$ be a projective, normal scheme of pure dimension $d$ over $k$, such that $(1-p^g)K_Y$ is Cartier for some integer $g >0$. Further let $A$ be a globally generated ample Cartier divisor on $Y$ and $H$ another Cartier divisor such that $H - dA- K_Y$ is ample.  Then for every $n \gg 0$, $\sHom_Y \left( F_{*}^{gn} \sO_{Y}, H \right)$ is globally generated.
\end{proposition}

\begin{proof}
By \cite[Theorem 1.8.5]{Lazarsfeld_Positivity_in_algebraic_geometry_I}, we are supposed to prove that for every $i >0$,
\begin{equation*}
H^i \left(Y, \sHom_Y \left( F_{*}^{gn} \sO_{Y}, H\right) \otimes \sO_Y(-iA) \right) = 0.
\end{equation*}
Note now that
\begin{multline*}
H^i \left(Y, \sHom_Y \left( F_{*}^{gn} \sO_{Y}, H\right) \otimes \sO_Y(-iA) \right) 
\\
\begin{split}
& \cong H^i\left(Y, \sHom_Y \left( F_{*}^{gn} \sO_{Y}, \sO_Y(H-iA)  \right) \right) \\
& \cong  \underbrace{H^i  \left(Y, F_{*}^{gn} \sHom_Y \left(  \sO_{Y}, \sO_{Y}(p^{gn}(H-iA) + (1-p^{gn}) K_{Y} )  \right) \right) }_{\textrm{Grothendieck duality}} \\
& \cong  \underbrace{H^i \left(Y, H-iA + (p^{gn}-1)(H-iA - K_{Y} ) \right) }_{\textrm{$F$ is affine}},
\end{split}
\end{multline*}
where the last group is zero for $n \gg 0$ by Serre-vanishing. 
\end{proof}

%

Recall that according to \autoref{sec:notation}, a fibration is a surjective, separable morphism of projective varieties over $k$ such that $f_* \sO_X \cong \sO_Y$.

\begin{lemma}
\label{lem:generic_Kodaira_dimension_behavior_by_an_ample_twist_from_base}
If  $f : X \to Y$ is a fibration and $A$ is a Cartier divisor on $X$, such that $X$ is normal, $\kappa(A_\xi) =l \geq 0$ (where $\xi$ is the generic point of $Y$) and $H$ is an ample Cartier divisor on $Y$, then $\kappa(A +  v f^*H) \geq  l + \dim Y$ for some (and hence every) $v \gg 0$.
\end{lemma}

\begin{proof}
Fix then an integer $s > 0$ and define the linear subspace $V_{n} \subseteq H^0(X_\xi,s n A_\xi)$ as 
\begin{equation*}
V_n:=\im ( \Sym^n H^0(X_\xi,sA_\xi) \to H^0(X_\xi,ns A_\xi)). 
\end{equation*}
 Further  choose this integer $s$, such that  $\kappa( V_\bullet) = l$, that is, $\dim V_n \geq c n^l$ for some positive real number $c>0$. Such choice of $s$ exists, because $sA_\xi$  is a Cartier divisor on the normal variety $X_\xi$ over the field $k(\xi)$ such that $H^0(X_\xi, \sO_{X_\xi})= k(\xi)$, and then \autoref{lem:embedding} applies (i.e., the fact that the pullback of $A_{\xi}$ to some birational modification contains the pullback of an ample line bundle). Note that the normality of $X_\xi$ follows from the normality of $X$, since $X_\xi$ is obtained from $X$ by localization.

Consider next $f_* \sO_X (sA)$. Since $H$ is ample, for every integer $t \gg 0$, $f_* \sO_X (sA) \otimes \sO_Y(tH) \cong f_* \sO_X(sA + tf^*H)$ is globally generated. Hence there is an injection
\begin{equation*}
H^{\oplus h^0(X_\xi, sA_\xi)} \hookrightarrow f_* \sO_X(sA + (t+1)f^*H) ,
\end{equation*}
which induces the following commutative diagram. 
\begin{equation*}
\xymatrix{
\left(H^{\otimes n}\right)^{\oplus \dim_{k(\xi)} \Sym^n H^0(X_\xi, sA_\xi)}  \cong S^{n} \left(H^{\oplus h^0(X_\xi, sA_\xi)}\right) \ar@{^(->}[r] & S^{n}\big(f_* \sO_X(sA + (t+1)f^*H)\big) \ar[d] \\ 
(H^{\otimes n})^{\oplus \dim_{k(\xi)} V_n} \ar@{^(->}[u]  \ar@{^(->}[r] & f_* \sO_X \big(n(sA + (t+1)f^*H)\big) &  \\ 
}
\end{equation*}
Hence, by the bottom horizontal line of the above diagram, for some positive real constant $c'$,
\begin{equation*}
h^0(n(sA + (t+1)f^* H)) \geq h^0(Y, nH) \cdot \dim_{k(\xi)} V_n \geq  c' n^{\dim Y} c n^l = c'  c n^{l + \dim Y}.
\end{equation*}
Then $\kappa(sA + (t+1) f^* H) \geq l + \dim Y$ by \autoref{cor:Kodaira_Iitaka_dimension_with_sections}.
\end{proof}

\begin{lemma}
\label{lem:from_effective_to_big}
If $f : X \to Y$ is a fibration, $A$ is a Cartier divisor on $X$ such that $X$ is normal, $\kappa(A_\xi)=l \geq 0$ (where $\xi$ is the generic point of $Y$) and $H$ an ample Cartier divisor on $Y$, then 
\begin{equation*}
\inf \{ s \in \bQ | \kappa( A + s f^*H) \geq l + \dim Y \} = \inf \{ s \in \bQ | \kappa(A + s f^* H) \geq 0  \}.
\end{equation*}
\end{lemma}

\begin{proof}
First, note that by definition
\begin{equation*}
t:=\inf \{ s \in \bQ | \kappa(A + s f^* H) \geq 0 \} \leq \inf \{ s \in \bQ | \kappa( A + s f^*H ) \geq l + \dim Y \}.
\end{equation*}
So, we are only supposed to prove the inequality in the other direction. By \autoref{lem:generic_Kodaira_dimension_behavior_by_an_ample_twist_from_base}, there is an $\bQ \ni a \gg 0$, such that $\kappa(A +af^*H) \geq l  + \dim Y$. Further, for every $\bQ \ni \varepsilon >0$,  $A + (t + \varepsilon) f^* H \sim_{\bQ} \Gamma$ for some effective $\bQ$-Cartier divisor $\Gamma$. Therefore the Kodaira-Iitaka dimension of the following divisor is at least $l + \dim Y$.
\begin{equation*}
(1- \varepsilon) (A  + (t +\varepsilon) f^* H) + \varepsilon (A + a f^*H)= A + ((1-\varepsilon) (t +\varepsilon) + \varepsilon a) f^*H
\end{equation*}
This concludes the proof, since $\lim_{\varepsilon \to 0} ((1-\varepsilon) (t+\varepsilon) + \varepsilon a) = t$.
\end{proof}

Now recall that for any scheme $X$ over $k$, $X^n$ denotes the source of the $n$-th iterated absolute Frobenius morphism of $X$, when it is important to distinguish between the source and the target. See \autoref{sec:notation} for an elaboration on this issue.

\begin{lemma}
\label{lem:effectivity_and_Frobenius_pullbacks}
Let  $f : X \to Y$ be a fibration such that $Y$ is normal and $K_Y$ is $\bQ$-Cartier with index not divisible by $p$. Further let $H$ be an ample divisor on $Y$. Then, for every $\bQ$-Cartier divisor $D$ on $X$ if $\kappa(D_{X_{Y^n}}) \geq 0$  for some $n \geq 0$, then $ \kappa(D + \varepsilon f^*H) \geq 0$ as well for every $ \bQ \ni \varepsilon >0$.
\end{lemma}

\begin{proof}
Since the statement of the proposition is invariant under scaling $H$, by \autoref{prop:base_globally_generated} we may assume that $\sHom_Y ( F_{Y,*}^n \sO_{Y^n}, H)$ is globally generated for every integer $n \geq 0$. Denote by $q$ the natural projection $X_{Y^n} \to X$. Since  $F_{Y,*}^n \sO_{Y^n}$ is  locally free over some non-empty open set $U$ of $Y$, it follows that 
\begin{equation*}
\sHom_X ( q_* \sO_{X_{Y^n}}, f^* H )|_{f^{-1}(U)} \cong f^* \sHom_Y ( F_{Y,*}^n \sO_{Y^n}, H )|_{f^{-1}(U)}.
\end{equation*}
In particular, $\sHom_X ( q_* \sO_{X_{Y^n}}, f^* H )$ is globally generated over $f^{-1}(U)$. Let $r >0$ be an integer such that $rD$ is an integer Cartier divisor and further that $rD_{X_{Y^n}}$ is linearly equivalent to an effective Cartier divisor. Then $\sO_{X_{Y^n}}(jrD_{X_{Y^n}} )$ has a section $s_j$ for every integer $j>0$. Choose $P \in X_{U^n} (:= X \times_Y U^n)$ such that $(s_j)_P \neq 0$ and let $Q$ be the image of $P$ in $X$ (which agrees with $P$ if we identify the underlying topological spaces for $X$ and $X_{Y^n}$). 
Note now that $s_j$ induces a section $\overline{s_j}$ of
\begin{equation*}
q_* \sO_{X_{Y^n}}(jrD_{X_{Y^n}}) \cong \sO_X(jrD) \otimes q_* \sO_{X_{Y^n}},
\end{equation*}
which is not zero at $Q$. Choose now an element $u$ of $\sHom(q_* \sO_{X_{Y^n}}, \sO_X) \otimes k(Q)$ that takes $(\overline{s_j})_Q$ to a non-zero element, and let $s \in \Hom_X ( q_* \sO_{X_{Y^n}}, f^* H )$ be an extension of $u$ to a global section. Then $\id_{\sO_X(jrD)} \otimes s$ takes $\overline{s_j}$ to a section of $\sO_X(jrD + H)$ which is not zero at $Q$. Hence $jrD + H$ is linearly equivalent to an effective Cartier divisor for every integer $j>0$. This concludes our proof.
\end{proof}

\section{Argument}
\label{sec:argument}

\begin{notation}
\label{notation:fibration}
Let  $f : X \to Y$ be a fibration (see \autoref{sec:notation}) with $\eta$ and $\overline{\eta}$ the perfect and the algebraic closures of the geometric general point of $Y$. Further assume that
\begin{enumerate}
\item $X$ is Gorenstein,
\item $Y$ is regular 
\item $X_{\overline{\eta}}$ is integral and
\item $S^0(X_\eta, \omega_{X_{\eta}}) \neq 0$. 
\end{enumerate}
We set $d:= \dim Y$. Further, fix a very ample Cartier divisor $A$ on $Y$. Let $H$ be a very ample Cartier divisor, such that  $H - K_Y - (d+1)A$ is very ample and further 
\begin{equation*}
 \sHom_Y \left( F_{Y,*}^{n} \sO_{Y^{n}}, H \right)
\end{equation*}
is globally generated for every $n \geq 0$ (such choice is possible by \autoref{prop:base_globally_generated}). Define then
\begin{equation*}
t:=\inf \{ s \in \bQ | \kappa( K_{X/Y} + s f^* H ) \geq 0 \}. 
\end{equation*}
\end{notation}

Our goal is to prove that in the situation of \autoref{notation:fibration}, $t \leq 0$. This implies then the statement of \autoref{thm:effective_general_fiber_implies_inf_is_zero} (to see that separability of $f$, normality of $X$ and that $f_* \sO_X \cong \sO_Y$ implies that $X_{\overline{\eta}}$ is integral use \cite[Prop 4.5.9, Prop 4.6.1]{Grothendieck_Elements_de_geometrie_algebrique_IV_II}). From now, we assume  \autoref{notation:fibration} throughout the section.

\begin{lemma}
For any integer $n >0$, if one replaces $f : X \to Y$ by $ f_{Y^n} : X_{Y^n} \to Y^n$ then the assumptions of Notation \ref{notation:fibration} still hold. 
\end{lemma}

\begin{proof}
Since $Y$ is regular, $Y^n \to Y$ is flat. Hence:
\begin{enumerate}
\item Since $X$ is Gorenstein $\omega_X^\bullet$ is a line bundle shifted in cohomological degree $- \dim X$. Using that $Y$ is Gorenstein as well, $\omega_{X/Y}^\bullet \cong \omega_X^\bullet [ - \dim Y] \otimes \omega_Y^{-1}$ and hence it is also a line bundle but it lives in cohomological degree $\dim Y - \dim X$.  Further, $\left( \omega_{X/Y}^\bullet \right)_{Y^n} \cong \omega_{X_{Y^n}/Y^n}^\bullet$  by flat pullback \cite[Corollary VII.3.4.a]{Hartshorne_Residues_and_duality}. So, $\omega_{X_{Y^n}}^\bullet \cong \omega_{X_{Y^n}/Y^n}^\bullet [ \dim Y] \otimes \omega_{Y^n} $ is also a line bundle shifted in cohomological degree $- \dim X$. Therefore,  $X_{Y^n}$ is indeed Gorenstein. 
\item $Y^n$ is regular, since as a scheme it is identical to $Y$.
\item The perfect and geometric generic points of $Y$ and of $Y^n$ agree. Hence the conditions on $X_{\eta}$ and $X_{\overline{\eta}}$ still hold.
\end{enumerate}
Further, we have to prove that $f$ is still a fibration. Separability is stable under pullback and surjectivity of $f_{Y^n}$ is also immediate since $f_{Y^n}$ agrees with $f$ on the underlying topological spaces. Further $f_{Y^n,*} \sO_{X_{Y^n}} \cong \sO_{Y^n}$ by flat base-change \cite[Proposition III.9.3]{Hartshorne_Algebraic_geometry}. So we only have to show that $X_{Y^n}$ is a variety. First, again by topological arguments we know that it is irreducible. We just have to prove that it is reduced. Since it is Cohen-Macaulay (and in particular $S_1$), its only embedded points can be its generic point. So, it is enough to show that $X_{Y^n}$ is reduced at its generic point. However, that follows from the separability of $f$.
\end{proof}

%
%
%
%

\begin{notation}
\label{notation:the_Cartier_module}
In the situation of \autoref{notation:fibration},
 assume that for some positive integers $n$, $q$ and $l$, $0 \neq \Delta \in |q K_{X_{Y^n}/Y^n} + f_{Y^n}^* l  H |$. Set $m:= q+1$.  Define then for every integer $n>0$ the Cartier module (which fact is proven in \autoref{lem:we_indeed_define_a_Cartier_module}).
\begin{equation*}
M_{\Delta} := S^0 f_{Y^n,*} (\sigma (X_{Y^n}, \Delta) \otimes \sO_X(K_{X_{Y^n}} + q K_{X_{Y^n}/Y^n} + f_{Y^n}^*lH)).
\end{equation*}
Here $H$ is regarded as a divisor on $Y^n$ via the natural $p^e$-linear isomorphism $Y \cong Y^n$ (or in other words via the identification $Y = Y^n$ obtained by forgetting the $k$-structures).
\end{notation}

According to \autoref{thm:openness_S_0}, $M_{\Delta}$ has rank at least as big as $S^0 \left(X_{\eta}, \sigma \left(X_{\eta}, \Delta_{\eta} \right) \otimes \sO_{X_{\eta}} \left(mK_{X_{\eta}} \right) \right)$. Hence our next task is to show that this is not zero. 

\begin{proposition}
\label{prop:non_vanishing_on_the_general_fiber}
With notation as above, $S^0 \left(X_{\eta}, \sigma \left(X_{\eta}, \Delta_{\eta} \right) \otimes \sO_{X_{\eta}} \left(mK_{X_{\eta}} \right) \right) \neq 0$.
\end{proposition}

\begin{proof}
First note that by the assumptions of \autoref{notation:fibration}, $S^0(X_\eta, \omega_{X_\eta}) \neq 0$. So, choose an element $0 \neq g$ of $S^0(X_{\eta}, \omega_{X_{\eta}})$. By definition for every integer $e \geq 0$ there is a $g_e \in H^0(X_{\eta}, \omega_{X_{\eta}})$ such that $\Tr (g_e)= g$. 
Denote by $h$ the element of $H^0(X_{\eta}, \omega_{X_{\eta}}^{m-1})$ corresponding to $\Delta_{\eta}$. Then we claim that
\begin{equation}
\label{eq:non_vanishing_on_the_general_fiber:elemnt_of_S_0}
gh \in S^0  (X_{\eta}, \sigma (X_{\eta}, \Delta_{\eta}) \otimes \omega_{X_{\eta}}^m ).
\end{equation}
Indeed,  to show \autoref{eq:non_vanishing_on_the_general_fiber:elemnt_of_S_0}, it is enough to show that when $g_e h \in H^0 \left(X_{\eta},\omega_{X_{\eta}}^m \right)$ is multiplied with $h^{p^e -1}$ and then the trace map is applied to it, then its image is $gh$.  This is done in the following computation:
\begin{equation*}
g_e h \mapsto g_e h h^{p^e -1} = \underbrace{g_e h^{p^e} \mapsto^{\Tr^e} gh}_{\textrm{$\Tr^e$ is $1/p^e$-linear}} . 
\end{equation*} 
\end{proof}

As we already mentioned, combining \autoref{prop:non_vanishing_on_the_general_fiber} and \autoref{thm:openness_S_0} we obtain the following.

\begin{corollary}
\label{cor:Cartier_module_non_zero}
In the situation of \autoref{notation:the_Cartier_module}, $M_\Delta  \neq 0$.
\end{corollary}

\begin{proposition}
\label{prop:effective_induction_step}
In the situation of \autoref{notation:the_Cartier_module}, $|(q+1)K_{X_{Y^n}/Y^n} + f_{Y^n}^*(l+2)H| \neq 0$.
\end{proposition}

\begin{proof}
By \autoref{thm:Cartier_module_globally_generated} and the choice of $H$ in \autoref{notation:fibration},   
\begin{equation*}
M_\Delta \otimes \sO_{Y^n}( H)
\cong \sO_{Y^n}((l + 1) H + K_{Y^n}) \otimes S f_{Y^n,*} \Big(\sigma \big(X_{Y^n}, \Delta \big) \otimes \sO_X \big(mK_{X_{Y^n}/Y^n} \big) \Big) 
\end{equation*}
is globally generated (the above isomorphism follows straight from the projection formula). 
Hence, we may choose a non-zero global section $t$ of $M_\Delta \otimes \sO_{Y^n}( H)$. Since
\begin{equation*}
M_\Delta \otimes \sO_{Y^n}( H) \subseteq  \sO_{Y^n}((l + 1) H + K_{Y^n}) \otimes f_{Y^n,*} \omega_{X_{Y^n}/Y^n}^m ,
\end{equation*}
 this induces a non-zero section of   $m K_{X_{Y^n}/Y^n} + f_{Y^n}^* ((l + 1) H + K_{Y^n})$ over $X_{Y^n}$. In particular, 
\begin{equation*}
|mK_{X_{Y^n}/Y^n} +  f_{Y^n}^*(l + 1) H + f_{Y^n}^* K_{Y^n} | \neq 0.
\end{equation*}
However, then since $H - K_Y$ is effective, 
$|m K_{X_{Y^n}/Y^n} +  f_{Y^n}^*(l + 2) H| \neq 0$ as well.
\end{proof}

\begin{proof}[Proof of \autoref{thm:effective_general_fiber_implies_inf_is_zero}]
Since the statement is invariant under scaling $L$, we may assume we are in the situation of \autoref{notation:fibration} and we may replace $L$ by $H$. Choose positive integers $a$ and $b$ and an effective divisor $|b K_{X/Y} + f^* a  H | \ni \Gamma \neq 0$. Such $a$, $b$ and $\Gamma$ exist by \autoref{lem:generic_Kodaira_dimension_behavior_by_an_ample_twist_from_base} and by the assumption of \autoref{notation:fibration} that $\omega_{X_\eta}$ has a section. Then $\Gamma_{Y^n} \in |b K_{X_{Y^n}/Y^n} + f_{Y^n}^* a p^n H |$. Using \autoref{prop:effective_induction_step} we see that $|(b+r)  K_{X_{Y^n}/Y^n} + f_{Y^n}^* (a p^n + 2r)  H | \neq 0$  for every integer $r \geq 0$. Hence by \autoref{lem:effectivity_and_Frobenius_pullbacks}, for every integer $n,r >0$,
\begin{equation*}
\kappa \left(E_{r,n} :=(b+r)  K_{X/Y} + f^* \left(a  + \frac{2r+1}{p^n} \right)  H \right) \geq 0.
\end{equation*}
Now, fix $n:=2v$ and $r:=p^v$. Then we see that
\begin{multline*}
\lim_{v \to \infty}  \frac{1}{b+p^v} E_{p^v,2v} 
= \lim_{v \to \infty} \frac{1}{b+p^v} \left(  (b+p^v)  K_{X/Y} + f^* \left(a  + \frac{2p^v+1}{p^{2v}} \right)  H  \right)
\\ = \lim_{v \to \infty}   K_{X/Y} + f^* \left(\frac{a}{b+p^v}  + \frac{2p^v+1}{p^{2v}(b+p^v)} \right)  H  = K_{X/Y}
\end{multline*}
Hence $t =0$ indeed. 
\end{proof}

\begin{proof}[Proof of \autoref{cor:subadditivity_of_Kodaira_dimension}]
Fix any ample Cartier divisor $L$ on $Y$. Since $K_Y$ is big, there is a rational number $\varepsilon>0$ and an integer $r>0$ such that $\varepsilon r$ is an integer and $rK_Y \sim r \varepsilon L + E$ for some effective Cartier divisor $E$ on $Y$. 
Further according to \autoref{thm:effective_general_fiber_implies_inf_is_zero} and \autoref{lem:from_effective_to_big}, by possibly multiplying $r$ and scaling $E$ accordingly,  there is also an effective Cartier divisor $\Gamma$ on $X$, such that $\Gamma \sim rK_{X/Y} + \varepsilon r f^* L$ and $\kappa(\Gamma) \geq \kappa(K_{X_\eta}) + \dim Y$. Therefore, 
\begin{equation}
\label{eq:subadditivity_of_Kodaira_dimension:linear_equivalence}
 r K_X \sim  r K_{X/Y} + r f^* K_Y \sim \left(\Gamma - \varepsilon r f^* L \right)  +  \left( r \varepsilon f^* L + f^* E \right) = \Gamma +   f^* E .
\end{equation}
Since  $E$ is effective, the following computation concludes our proof. 
\begin{equation*}
\kappa(K_X) = \kappa(\Gamma + f^* E) \geq \kappa(\Gamma)   \geq \kappa(K_{X_\eta}) + \dim Y
\end{equation*}
\end{proof}

\section{Weak-positivity}
\label{sec:weak_positivity}

Here we show a weak positivity statement, \autoref{thm:weak_positivity}, that is not needed for the main statements of the paper, but it is philosophically closely connected as explained after the statement of \autoref{cor:subadditivity_of_Kodaira_dimension}. Note that \autoref{thm:weak_positivity} is, according to the best knowledge of the author, the first positive characteristic variant of the famous weak-positivity theorem of Viehweg \cite[Thm 4.1]{Viehweg_Weak_positivity}.

\begin{proposition}
\label{prop:weak_positivity}
Let $f : X \to Y$ be a fibration with $X$ Gorenstein and $Y$ regular. Then
\begin{equation*}
S^0 f_{Y^{n},*} \omega_{X_{Y^n}/Y^{n}} \subseteq \left(S^0 f_*   \omega_{X/Y} \right)_{Y^{n}}
\end{equation*}
as subsheaves of $(f_*(M))_{Y^{n}}$. Furthermore, 
\begin{equation*}
\mathrm{Tr}^{n} F_{Y,*}^{n} S^0 f_{Y^{n},*} \omega_{X_{Y^n}/Y^{n}} = S^0 f_*  \omega_{X/Y}.
\end{equation*} 

\end{proposition}

\begin{proof}
The proof is identical to that of \cite[Proposition 6.36]{Patakfalvi_Schwede_Zhang_F_singularities_in_families} using that for Gorenstein morphisms $\omega_{X/Y}$ is compatible with base-change and that $S^0 f_*  \omega_{X/Y}$ stabilizes by \autoref{prop:Cartier_module_stabilizes}.
\end{proof}

\begin{corollary}
\label{cor:rank_stabilization}
With notations as above, $r(n):=\rk S^0 f_{Y^{n},*} \omega_{X_{Y^n}/Y^{n}}$ is a (not necessarily strictly) decreasing function of $n$ and further if $S^0 f_* \omega_{X/Y} \neq 0$, then $r(n)>0$ for all $n>0$. Hence, there exists a greater than zero minimum of $r(n)$, which we denote by $r$. Let $n_0$ be the smallest  value of $n$ such that $r = r(n)$.
\end{corollary}

We give now the definition of weak-positivity, which is the slightly weaker version used in \cite[Notation, (vii)]{Kollar_Subadditivity_of_the_Kodaira_dimension}, rather than that in  \cite[Definition 2.11]{Viehweg_Quasi_projective_moduli}. Note that this version is also called pseudo-effectivity in \cite[Proposition 6.3]{Demailly_Peternell_Schneider_Pseudo_effective_line_bundles_on_compact_Kahler_manifolds}.

\begin{definition}[Definition 2.11 of \cite{Viehweg_Quasi_projective_moduli}]
\label{defn:weakly_positive}
Let $\sF$ be a torsion-free  sheaf on a quasi-projective (over a field $k$) variety $V$. Then $\sF$ is  \emph{weakly positive}, if for a fixed (or equivalently every: \cite[Lemma 2.14.a]{Viehweg_Quasi_projective_moduli}) ample line bundle $\sH$ for every integer $a>0$ there is an integer $b>0$ such that $S^{[ ab]} (\sF) \otimes  \sH^b$ is generically globally generated.
\end{definition}

\begin{theorem}
\label{thm:weak_positivity}
If $f : X \to Y$ is a fibration with $X$ Gorenstein and $Y$ regular, then  $S^0 f_{Y^{n},*} \omega_{X_{Y^{n}}/Y^{n}}$ is weakly-positive for every integer $n \geq n_0$ (where $n_0$ is as defined in \autoref{cor:rank_stabilization}). Further, $S^0 f_* \omega_{X/Y} \neq 0$ or $S^0(X_\eta, \omega_\eta) \neq 0$, where $\eta$ is the perfect general point, then  $S^0 f_{Y^{n},*} \omega_{X_{Y^{n}}/Y^{n}} \neq 0$.
\end{theorem}

\begin{proof}
We use the definitions made in \autoref{cor:rank_stabilization}.
Fix a globally generated ample line bundle $A$ on $Y$ and let $H$ be a very ample Cartier divisor such that $H- K_Y - (\dim Y) A$ is also very ample. Further introduce
\begin{equation*}
 \sF_n:=  S^0 f_{Y^{n},*} \omega_{X_{Y^{n}}/Y^{n}}.
\end{equation*}
Fix an integer $a>0$ and choose $m > n$, such that $2a \leq p^{m-n}$. Then, by  \autoref{thm:Cartier_module_globally_generated}, $\sF_m \otimes H \otimes \omega_{Y^m}$ is globally generated on $Y^m$. However, then so is $\sF_m \otimes H^{\otimes 2}$. Therefore by the natural injection of full rank  $F^{m-n,*} \sF_{n} \otimes H^{2}$ is generically globally generated as well. Hence so is $F^{m-n,*} S^{[a]}(\sF_{n}) \otimes H^{2a}$ (note that since $F_Y$ is flat, the pullback of a reflexive sheaf by it is also reflexive). However than since $2a \leq p^{m-n}$, 
\begin{equation*}
F^{m-n,*} S^{[a]}(\sF_{n}) \otimes H^{p^{m-n}} \cong F^{m-n,*} (S^{[a]}(\sF_{n}) \otimes H) 
\end{equation*}
is also generically globally generated. Then by \autoref{prop:base_globally_generated}, $S^{[a]}(\sF_{n}) \otimes H^{\otimes 2}$ is generically globally generated on $Y^{n}$. Since this is true for all integers $a \geq 0$, $\sF_{n}$ satisfies the definition of weak-positivity.

To see the first case of the addendum, note that if $S^0 f_* \omega_{X/Y} \neq 0$ then $r >0$. To see the second case, use \autoref{thm:openness_S_0}.
\end{proof}

\appendix

\section{Kodaira-Iitaka dimension over arbitrary field}
\label{sec:Kodaira_Iitaka_dimension}

Recall that a variety is an integral, separated scheme over a field $k_0$. We are not using $k$ for the base-field, because that denotes a fixed perfect base-field. We definitely want to allow non-perfect base-fields in this section.

\begin{notation}
\label{notation:Kodaira_Iitaka}
Fix a proper variety $X$ over a field $k_0$ such that $H^0(X, \sO_X)= k_0$ and $L$ a line bundle on $X$. 
\end{notation}

\begin{definition}
\label{defn:Kodaira_Iitaka}
 In the situation of \autoref{notation:Kodaira_Iitaka}, with further also allowing $L$ to be $\omega_X$ if $X$ is $S_2$ and $G_1$,  we define
\begin{equation*}
R(X,L):= \bigoplus_{m \geq 0} H^0\left(X, L^{[m]} \right),
\end{equation*}
where $L^{[m]}:=(L^{\otimes m})^{\vee \vee}$. Note that $R(X,L)$ is a $\bZ$-graded integral domain, and the double dual can be disregarded unless $L= \omega_X$ and $\omega_X$ is not a line bundle. The field of  degree zero elements of the fraction field of $R(X,L)$ is denoted by $Q(X,L)$. Note that $Q(X,L)$ comes with a natural embedding into the fraction field $Q(X)$ of $X$. We define the Kodaira-Iitaka dimension $\kappa(L)$ of $L$ to be
\begin{equation*}
\kappa(L)= \left\{
\begin{array}{ll}
- \infty & \textrm{if } R(X,L) = k_0 \\
\trdeg_{k_0} Q(X,L) & \textrm{otherwise.}
\end{array}
\right.
\end{equation*}
Further define $S(R(X,L)_n)$ to be the $k_0$-algebra generated by the degree $n$ homogeneous part $R(X,L)_n$ of $R(X,L)$. Denote by $Q(X,L)_n$ the field of degree zero elements of the fraction field of $S(R(X,L)_n)$. We denote by $\phi_L$ the map defined by $L$.
\end{definition}

 We denote by $Q(Z)$ the field of rational functions for any variety $Z$, by $\phi_M$ the rational map induced by $M$ (the map on the complement of the base-locus of $M$ defined in \cite[Thm II.7.1]{Hartshorne_Algebraic_geometry}). Then the next lemma follows from the definitions. 

\begin{lemma}
\label{lem:Kodaira_dimension_map}
In the situation of \autoref{notation:Kodaira_Iitaka}, $Q(\phi_{mL}(X)) = Q(X,L)_m$ for every integer $m >0$.
\end{lemma}


\begin{lemma} \cite[1.2.i]{Mori_Classification_of_higher_dimensional_varieties}
In the situation of \autoref{notation:Kodaira_Iitaka}, $Q(X,L)_m = Q(X,L)$ for every integer $m$ divisible enough.
\end{lemma}

\begin{corollary}
In the situation of \autoref{notation:Kodaira_Iitaka}, $Q(\phi_{mL}(X)) = Q(X,L)$ for every integer $m$ divisible enough. In particular, an equivalent definition of Kodaira-Iitaka dimension is
\begin{equation*}
\kappa(L)= \left\{
\begin{array}{ll}
- \infty & \textrm{if } h^0(X,mL)=0 \textrm{ for every } m>0 \\
\max\{\dim \left( \phi_{mL}(X) \right)|m \geq 0, m \in \bZ\} & \textrm{otherwise.}
\end{array}
\right.
\end{equation*}
Furthermore, in the second case, the maximum is attained for all $m$ divisible enough.
\end{corollary}

\begin{corollary}
\label{cor:Kodaira_Iitaka_dimension_base_extension}
In the situation of \autoref{notation:Kodaira_Iitaka}, if $k_0 \subseteq K$ is a field extension, $\kappa(L_K) = \kappa(L)$.
\end{corollary}

\begin{proof}
By flat base-change $H^0(X_K,mL_K) \cong H^0(X, mL) \otimes_{k_0} K$. In patricular, since $k_0 \to K$ is faithfully flat, $\kappa(L)= - \infty$ if and only if $ \kappa(L_K) = - \infty$. So, assume that $\kappa(L) \neq \infty$. Then by the above flat base-change, $\phi_{mL_K}$ agrees with the base-change of $\phi_{mL}$. In particular then $\phi_{mL_K}(X_K) = \phi_{mL}(X) \otimes_{k_0} K$. Therefore $\dim \phi_{mL_K}(X_K) = \dim \phi_{mL}(X)$, which concludes our proof.
\end{proof}

\begin{lemma}
\label{lem:elimination_of_indeterminancies}
In the situation of \autoref{notation:Kodaira_Iitaka} assuming also that $X$ is normal and $h^0(mL) >0$, there is a birational morphism $g : Z \to X$ from a normal variety $Z$ over $k_0$ and a morphism $f : Z \to Y$, such that the following diagram commutes
\begin{equation*}
\xymatrix{
X \ar@{-->}[d]_{\phi_{mL}} & Z \ar[l]_g \ar[dl]^f \\ \
Y
}.
\end{equation*}
Furthermore if $m$ divisible enough then $Q(Y) = Q(X,L)= Q(X,L)_n$ for every integer such that $m | n$.
\end{lemma}

\begin{proof}
Note that the addendum follows from \autoref{lem:Kodaira_dimension_map} as soon as we construct a commutative diagram as in the statement. For that let $Z$ be the normalization of the main component of the closure of the graph of $\phi_{mL}$ in $X \times Y$. Then $g: Z \to Y$ is automatically birational, because it is isomorphism over the domain of $\phi_{mL}$.
\end{proof}

\begin{lemma} \cite[Claim in the proof of 1.11]{Mori_Classification_of_higher_dimensional_varieties}
\label{lem:embedding}
With notation as in \autoref{lem:elimination_of_indeterminancies}, let $H$ be the  very ample divisor on $Y$ used to define $f$. Then there is an embedding $f^* \sO_X(H) \subseteq g^* L^m$.
\end{lemma}

\begin{lemma}
\label{lem:Kodaira_Iitaka_dimension_generic_fiber}
 With notation as in \autoref{lem:elimination_of_indeterminancies} (including the assumption that for all $m |n$, $Q(Y) = Q(X, L)_n$), $\kappa((g^* L)_{\xi}) = 0$, where $\xi$ is the generic point of $Y$.
\end{lemma}

\begin{proof}
So far in this section we have used the additive notation for $L$, or in other words, we regarded $L$ as a Cartier divisor. Now, we will have to choose the multiplicative notation, since the line bundle structure on $L$ is used at many places of the proof.  First note that for every integer $n \geq 0$ there is a natural homomorphism
\begin{equation*}
H^0(Z, g^* L^n) \to H^0(Z_\xi, (g^* L^n)_{\xi} ),  
\end{equation*}
which is an injection, because it is induced by a localization in an integral domain. Hence $\kappa \left((g^* L)_{\xi} \right) \geq 0$. To show that $\kappa \left((g^* L)_{\xi} \right) = 0$, by cohomology and base-change (i.e, the fact that base-change holds at the generic point over an integral base), it is enough to show that $f_* g^*L^n$ has rank at most one for every divisible enough integer $n \geq 0$. 

First, consider the natural map 
\begin{equation*}
\gamma :H^0(Y, f_*  g^* L^n ) \otimes_{k_0} \sO_Y    \to  f_*  g^* L^n
\end{equation*}
for any integer $n$ for which $m |n$.
Assume that the image of $\gamma $ has rank greater than one. This means that there are two sections $s$ and $t$ of $H^0(Z, g^* L^n) = H^0(X, L^n)$ that are $k(\xi)(=Q(Y) = Q(X,L))$-linearly independent as elements of $H^0(Z_\xi, (g^* L^n)_\xi)$. Therefore, $\frac{s}{t}$ is an element of $Q(Z) \setminus Q(Y) = Q(X) \setminus Q(X,L)$. However, this is a contraction, since $\frac{s}{t} \in Q(X,L)$ by definition of $Q(X,L)$ and the choice of $s$ and $t$. So, we have proved that for every integer  $n \geq 0$ for which $m |n$,  the image of $\gamma$ is at most rank one. 

Consider now the following commutative diagram .
\begin{equation*}
\xymatrix{
 H^0(Y, (f_*  g^* L^n) \otimes H^{b} ) \otimes_{k_0} \sO_Y \ar@{^(->}[d]    \ar[r]^-{\beta} & (f_*  g^* L^n) \otimes H^{b}  \ar@{^(->}[d] \\
H^0(Y, f_*  g^* L^{n + mb}) \otimes_{k_0} \sO_Y    \ar[r]^-{\gamma } &  f_* g^* L^{n + mb}. 
}
\end{equation*}
Fix $n$ first, such that $m|n$. Then, since $H$ is ample, for every $b$ big enough $\beta$ is surjective. So, since the image of $\gamma$ has rank at most one, $f_* g^* L^n$ has rank at most one as well. This concludes our proof.

\end{proof}

\begin{corollary}
\label{cor:Kodaira_Iitaka_asymptotic_characteriztaion_normal_varieties}
In the situation of \autoref{notation:Kodaira_Iitaka} assuming also that $X$ is normal and $\kappa(L) \geq 0$, there are real, positive constants $c$ and $d$ such that
\begin{equation*}
c n^{\kappa(L)} \leq h^0(L^n) \leq d n^{\kappa(L)}. 
\end{equation*}
for every divisible anough $n$.
\end{corollary}

\begin{proof}
In this proof we keep on using the line bundle notation. Note first that the construction of  \autoref{lem:elimination_of_indeterminancies} applies to $X$ and $L$.  So, we assume the notations of \autoref{lem:elimination_of_indeterminancies} (including the assumption that for all $m |n$, $Q(Y) = Q(X, L)_n$).

First, we show that $f_* \sO_Z \cong \sO_Y$. For that  note that since $mL$ is linearly equivalent to an effective divisor there is a natural embedding $\sO_Z \hookrightarrow g^* L^n$, which induces an embedding $f_* \sO_Z \hookrightarrow f_* g^* L^n$. Since the latter has rank one, so does the former. However then $f_* \sO_Z$ is a sheaf of integral domains which is generically isomorphic to its subsheaf $\sO_Y$ of integrally closed integral domains. Then it follows that $f_* \sO_Z = \sO_Y$.

By \autoref{lem:embedding}, there is an embedding $f^* H \hookrightarrow g^* L^m$ where $H$ is a very ample divisor on $Y$. This induces $f^* H^a \hookrightarrow g^* L^{ma}$. Hence there is an embedding $H^a \cong f_* f^* H^a \hookrightarrow f_* g^* L^{ma}$, which induces an embedding $H^0(Y, H^a) \hookrightarrow H^0(X, L^{am})$. Therefore $h^0(H^a) \leq h^0(L^{am})$, which shows the lower bound in the statement of the corollary.

We are left to show the upper bound. By \autoref{lem:Kodaira_Iitaka_dimension_generic_fiber}, $\kappa((g^*L)_\xi)=0$. So, choose one $n$ for which $h^0(nL) \neq 0$. Let $E' \in |n g^*L|$ be arbitrary and let $E$ be the horizontal part of $E'$ (i.e., the union of the components that map surjectively onto $Y$ (with the right coefficients)). First, note that for any open set $U \subseteq Y$ and  $F \in \left| a n \cdot g^* L|_{f^{-1}(U)} \right|$, the horizontal part of $F$ is  $aE|_{f^{-1}(U)}$, since $F_\xi= (aE')_\xi= (aE)_\xi$ by $\kappa((g^*L)_\xi)=0$. In particular, $f_* g^* L^{a \cdot n} \cong f_* (g^* L^{a \cdot n}(-aE))$ induced by the embedding $g^* L^{a \cdot n} (-aE) \to g^*L^{a \cdot n}$.  Let $\sum a_i E_i$ be the prime decomposition of $E' -E$ and $\eta_i$ the generic point of $E_i$. Choose then a very ample Cartier divisor $A$ on $Y$, such that the multiplicity of $A$ at every point $y \in Y$ is at least $\sum_{ f(\eta_i) = y } a_i$. Note that this number is non-zero only at 
finitely many points. Hence such a choice of $A$ is possible. Further, by the choice of $A$, $E' - E \leq f^* A$. So, there is an embedding $g^*L^n (-E) \hookrightarrow f^* A$, which induces $g^* L^{a \cdot n} (-aE) \hookrightarrow f^* A^a$ and then also the embedding $f_* g^* L^{a \cdot n} \cong f_* (g^*L^{a \cdot n}(-aE)) \hookrightarrow f_* f^* A^a \cong A^a$. In particular this implies the inequality $h^0(L^{a \cdot n})=h^0(g^* L^{a \cdot n}) \leq h^0(A^a)$, which concludes the upper bound as well.
 
\end{proof}

\begin{corollary}
\label{cor:Kodaira_Iitaka_dimension_with_sections}
In the situation of \autoref{notation:Kodaira_Iitaka}, such that $X$ is normal, $\kappa(L) \geq 0$ and $k_0 \subseteq K$ is a field extension for which $X_K$ is integral, then there are real, positive constants $c$ and $d$ such that
\begin{equation*}
c n^{\kappa(L_K)} \leq h^0(L_K^n) \leq d n^{\kappa(L_K)}. 
\end{equation*}
for every divisible enough $n$.
\end{corollary}

\begin{proof}
By \autoref{cor:Kodaira_Iitaka_dimension_base_extension}, $\kappa(L_K) = \kappa(L)$. Further, by flat base-change $\dim_K H^0(X_K,L_K^n)= \dim_{k_0} H^0(X, L^n)$. Then the statement follows from \autoref{cor:Kodaira_Iitaka_asymptotic_characteriztaion_normal_varieties}.
\end{proof}

\section{Kodaira dimenison versus the Kodaira-Iitaka dimension of the canonical bundle}
\label{appendix:Kodaira_dimension_normalization}

Recall that the canonical sheaf of a projective variety $X$ over $k_0$ such that $H^0(X, \sO_X) = k_0$ is defined as $\omega_{X/k_0}:=\sH^{- \dim X}(\nu^! \sO_{k_0})$, where $\nu : X \to \Spec k_0$ is the structure map. If $k_0$ is clear from the context, we often omit it from the notation: we write $\omega_X$ instead of $\omega_{X/k_0}$.  By flat base-change $\left(\omega_{X/k_0} \right)_K \cong \omega_{X_K/K}$ for any field extension $k_0 \subseteq K$. Hence if $\omega_{X/k_0}$ is a line bundle, then by \autoref{cor:Kodaira_Iitaka_dimension_base_extension} $\kappa \left(\omega_{X/k_0} \right) = \kappa \left(\omega_{X_K/K} \right)$ or shortly just $\kappa(\omega_X) = \kappa \left(\omega_{X_K} \right)$. In particular, $\kappa(\omega_X) = \kappa \left(\omega_{X_{\overline{k}_0}} \right)$.
The main purpose of this appendix is to show that  $\kappa \left(\omega_{X_{\overline{k}_0}} \right)$ is at least as big as  the Kodaira dimension $\kappa \left(X_{\overline{k}_0} \right)$, which stated in \autoref{cor:Kodaira_vs_Kodaira-Iitaka_of_canonical}. For the definition of the Kodaira dimension of a variety see \cite[Def 5.1]{Luo_Kodaira_dimension_of_algebraic_function_fields}. Note that we are forced to use \cite[Def 5.1]{Luo_Kodaira_dimension_of_algebraic_function_fields} instead of the usual definition with resolution of singularities, because resolutions are not known to exist in positive characteristic for varieties of dimension higher than 3. On the other hand, if there is a resolution of singularities $Y \to X_{\overline{k}_0}$, then one could define $\kappa \left(X_{\overline{k}_0} \right)$ the traditional way, that is, as $\kappa(\omega_Y)$. First, in \autoref{prop:Kodaira_vs_Kodaira-Iitaka_of_canonical} we show that the two definitions coincide (i.e, that $\kappa(\omega_Y)= \kappa(Y)$) and 
the relation between $\kappa \left(\omega_{X_{\overline{k}_0}} \right)$ and $\kappa \left(X_{\overline{k}_0} \right)$ in the special case when $X_{\overline{k}_0}$ is normal. Then in \autoref{prop:kappa_normalization} we treat the non-normal case.

\begin{proposition}
\label{prop:Kodaira_vs_Kodaira-Iitaka_of_canonical}
Let $X$ be a projective, normal variety over an algebraically closed field $k_1$. Then 
\begin{enumerate}
 \item \label{itm:Kodaira_vs_Kodaira-Iitaka_of_canonical:normal} $\kappa(X) \leq \kappa(\omega_X)$ and
\item \label{itm:Kodaira_vs_Kodaira-Iitaka_of_canonical:smooth} if $X$ is smooth over $k_1$, then $\kappa(X)=\kappa(\omega_X)$.
\end{enumerate}
\end{proposition}

\begin{proof}
We show the two statements at once. Denote by $K$ the function field of $X$. Let $U$ be the regular locus of $X$. In the following computation we denote $\omega_X^{\otimes m}$ by $\omega_X^m$, and for each $P \in U$, we regard $\omega_{X,P}^m$ as subsheaves $\omega_{X, \eta} (=\det_K \Omega_{K/k_1})$, where $\eta$ is the generic point of $X$.
\begin{equation*}
\begin{split}
H^0 \left(X, \omega_X^{[m]} \right)
& \cong \underbrace{H^0(U, \omega_X^{ m}) }_{\textrm{\cite[Thm 1.9]{Hartshorne_Generalized_divisors_on_Gorenstein_schemes}}}
 \cong \underbrace{\bigcap_{P \in U} \omega_{X,P}^{m}}_{\textrm{by definition}}
= \underbrace{\bigcap_{P \in U, \codim P = 1} \omega_{X,P}^m}_{\textrm{\cite[Thm 38, page 124]{Matsumura_Commutative_algebra}}}
\\ & = \underbrace{\bigcap_{P \in U, \codim P = 1} \sO_{X,P} \left(dx_1^P\wedge \dots \wedge dx_n^P\right)^{\otimes m}}_{x_1^P, \dots , x_n^P \textrm{ are local coordinates for some regular specialization of $P$ }}
\\ & = \underbrace{\bigcap_{P \in U, \codim P = 1} \sO_{X,P} \left(\sJ^P_{P_0}\right)^m \left(dx_1^{P_0} \wedge \dots \wedge dx_n^{P_0}\right)^{\otimes m}}_{P_0 \in U \textrm{ is fixed such that } \codim P_0=1 \textrm{ and } \sJ^P_{P_0} \textrm{ is explained below.}}
\end{split}
\end{equation*}
Here $\sJ^P_{P_0}$ is the determinant of the matrix $(a_{ij})$, where $a_{ij}$ are defined via the equations
\begin{equation*}
 dx_j^P = \sum_{i=1}^n a_{ij} dx_i^{P_0}
\end{equation*}
using that both $dx_1^P, \dots, dx_n^P$ and $dx_1^{P_0}, \dots, dx_n^{P_0}$ are bases of $\Omega_{K/k_1}$. Now, using the language of \cite{Luo_Kodaira_dimension_of_algebraic_function_fields}, choose a $k_1/k_1$-differential basis $B$ \cite[page 672]{Luo_Kodaira_dimension_of_algebraic_function_fields}, where $k_1$ is the prime field of $k_1$. Then $B_{R_P}:= B \cup \{x_1^P, \dots, x_1^P\}$ defines a set of normal uniformizing coordinates  for $R_P:=\sO_{X,P}$ \cite[Def 2.1]{Luo_Kodaira_dimension_of_algebraic_function_fields}. Hence using the language of \cite{Luo_Kodaira_dimension_of_algebraic_function_fields}, let $J'\left(B_R,B_{R_0} \right)$ be any lift of $J'\left(B_R,B_{R_0} \right)$ into $K^*$. Then we can rephrase our previous computation:
\begin{equation*}
\begin{split}
H^0 \left(X, \omega_X^{[m]} \right)
& \cong \underbrace{\bigcap_{R \textrm{ is a local ring of $X$ of dimension } 1}  J'\left(B_R,B_{R_0} \right)^m \cdot R x^m}_{\textrm{the intersection is taken inside the free $K$-module $K x^m$, and $R_0$ is a fixed local ring of $X$ of dimension $1$ }} 
\\ & = \underbrace{\bigcap_{R \textrm{ is a local ring of $X$ of dimension } 1}  
R \cdot (\tau(B_R))^m}_{\textrm{where, $\tau(B_R) = J'(B_R,B_{R_0})x$ as defined in \cite[Definition 4.2]{Luo_Kodaira_dimension_of_algebraic_function_fields} }}
\\ & \supseteq \bigcap_{R \textrm{ is a regular locality of } K/k_1} R \cdot (\tau(B_R))^m
\end{split}
\end{equation*}
This implies that $R(X, \omega_X) \supseteq C(K/k_1)$, where the latter is the canonical ring of $K$ over $k_1$ defined in \cite[Definition 4.3]{Luo_Kodaira_dimension_of_algebraic_function_fields}. This concludes the proof of point \eqref{itm:Kodaira_vs_Kodaira-Iitaka_of_canonical:normal}. For point \eqref{itm:Kodaira_vs_Kodaira-Iitaka_of_canonical:smooth}, just note that the containment in the last computation is equality if $X$ is smooth by \cite[Thm 4.7]{Luo_Kodaira_dimension_of_algebraic_function_fields}. 
\end{proof}

\begin{proposition}
\label{prop:kappa_normalization}
Let $X$ be a projective, $S_2, G_1$ variety over a field $k_0$ and $f: Y \to X$ its normalization. Then $\kappa(\omega_X) \geq \kappa(\omega_Y)$.
\end{proposition}

\begin{proof}
According to \autoref{defn:Kodaira_Iitaka}, it is enough to show that for every integer $m>0$, $h^0\left(Y, \omega_Y^{[m]}\right) \leq h^0\left(X, \omega_X^{[m]} \right)$. To prove this we may discard freely codimension two subvarieties of both $X$ and $Y$ \cite[Thm 1.9]{Hartshorne_Generalized_divisors_on_Gorenstein_schemes}. In particular, by dropping the projectivity assumption on $X$, we may assume that $Y$ is regular and $X$ is Gorenstein, and then in particular both $\omega_X$ and $\omega_Y$ are line bundles. We use these assumptions from now on.

\emph{First we claim that $ \omega_Y (E) \cong f^* \omega_X$ for some effective divisor $E$, that is, there is an embedding $\omega_Y \hookrightarrow f^* \omega_X$.} By Grothendieck-duality there is an embedding $f_* \omega_Y \to \omega_X$. This induces a homomorphism $\xi : f^* f_* \omega_Y \to f^* \omega_X$, which is isomorphism on the open set $U$, where  $f$ is isomorphism. In particular, then $\Ker \xi$ is torsion. Further, note that since $f^* \omega_X$ is torsion-free, every torsion element of $f^* f_* \omega_Y$ is contained in $\Ker \xi$. Therefore, $\Ker \xi$ equals the torsion submodule of $f^* f_* \omega_Y$. Consider now the natural homomorphism $\zeta : f^* f_* \omega_Y \to \omega_Y$. Since $f$ is finite, $\zeta$ is surjective. Furthermore, by the same reasons as above $\Ker \zeta$ is the torsion submodule of $f^* f_* \omega_Y$. Hence $\zeta$ factors $\xi$ and therefore we obtained a natural morphism $\omega_Y \to f^* \omega_X$, which is isomorphism on $U$. However then, since $Y$ is 
integral, it is an embedding. This finishes the proof of our claim.

So, as stated earlier, to conclude our proof we are supposed to prove that $h^0\left(Y, \omega_Y^m \right) \leq h^0(X, \omega_X^{m})$. For that  it is enough to show that there is a homomorphism $f_* \omega_Y^m \to \omega_X^m$ restricting to an isomorphism on $U$. Indeed, start with the Gorthendieck trace $f_* \omega_Y \to \omega_X$ and precompose it with $f_*( \omega_Y (-(n-1) E)) \to f_* \omega_Y$, obtaining this way $f_*( \omega_Y (-(n-1) E)) \to \omega_X$. Tensor this then with $\omega_X^{n-1}$. This yields a homomorphism 
\begin{equation*}
f_*( \omega_Y (-(n-1) E)) \otimes \omega_X^{(n-1)} \cong f_* \left( \omega_Y (-(n-1) E)\otimes f^* \omega_X^{(n-1)} \right)  \cong f_* (\omega_Y^m)  \to \omega_X \otimes \omega_X^{(n-1)}.
\end{equation*}
It follows from the construction that it is an isomorphism over $U$ and hence it is an embedding. 
\end{proof}

\begin{corollary}
\label{cor:Kodaira_vs_Kodaira-Iitaka_of_canonical}
If $X$ is an $S_2,G_1$ variety over an algebraically closed field, then $\kappa(X) \leq \kappa( \omega_X)$.
\end{corollary}

\bibliographystyle{skalpha}
\bibliography{includeNice}

\end{document}